\documentclass[10pt]{article}
\usepackage{amsmath,amstext,amsthm,a4,amssymb,amscd}
\usepackage[mathscr]{eucal}
\usepackage{mathrsfs}
\usepackage{bbold}
\usepackage{epsf}
\usepackage{typearea}
\usepackage{charter}
\usepackage{bm}
\usepackage{dsfont}

\usepackage{graphicx}
\usepackage[all]{xy}
\usepackage{hyperref}
\usepackage[usenames,dvipsnames]{xcolor}
\usepackage{xy}
\hypersetup{colorlinks=true,citecolor=NavyBlue,linkcolor=Brown,urlcolor=Orange}

\newcommand{\Q}{\mathbb{Q}}
\newcommand{\R}{\mathbb{R}}
\newcommand{\C}{\mathbb{C}}
\newcommand{\Z}{\mathbb{Z}}
\newcommand{\N}{\mathbb{N}}

\newcommand{\CP}{\C\mathrm{P}}
\newcommand{\Bl}{\mathrm{Bl}}

\newcommand{\smooth}{{\mathcal{C}^\infty}}

\newcommand{\im}{\mathrm{Im}}

\newcommand{\Gr}{\mathrm{Gr}}

\renewcommand{\d}{\mathrm{d}} 
\renewcommand{\div}{\mathrm{div}} 

\makeatletter
\newcommand*{\rom}[1]{\expandafter\@slowromancap\romannumeral #1@}
\makeatother

\DeclareMathOperator{\Gor}{Gor}
\DeclareMathOperator{\Kah}{K\ddot{a}h}

\DeclareMathOperator{\Var}{Var}

\DeclareMathOperator{\Hdg}{Hdg}
\DeclareMathOperator{\VHS}{VHS}
\DeclareMathOperator{\MHM}{MHM}

\newcommand{\taub}{\tau_d^\mathrm{bir}}

\newtheorem{prop}{Proposition}[section]
\newtheorem{thm}[prop]{Theorem}
\newtheorem{lemme}[prop]{Lemma}

\newtheorem{conj}[prop]{Conjecture}

\newtheorem{alphatheorem}{Theorem}

\theoremstyle{definition}
\newtheorem{defn}[prop]{Definition}
\newtheorem{ex}[prop]{Examples}

\theoremstyle{remark}
\newtheorem{rem}[prop]{Remark}

\numberwithin{equation}{section}
\providecommand{\keywords}[1]
{
	\small	
	\textbf{\textit{Keywords:}} #1
}

\title{Motivic integration \\
and birational invariance of BCOV invariants}

\author{Lie Fu and Yeping Zhang}
\date{}

\begin{document}
\maketitle
\date{}
\begin{abstract}
Bershadsky, Cecotti, Ooguri and Vafa constructed a real valued invariant for Calabi--Yau manifolds,
which is now called the BCOV torsion. Based on it, a metric-independent invariant, called BCOV invariant, was constructed by Fang--Lu--Yoshikawa and Eriksson--Freixas i Montplet--Mourougane.
The BCOV invariant is conjecturally related to the Gromov--Witten theory via mirror symmetry.
Based upon previous work of the second author, we prove the conjecture that birational Calabi--Yau manifolds have the same BCOV invariant.
We also extend the construction of the BCOV invariant
to Calabi--Yau varieties with Kawamata log terminal singularities, and prove its birational invariance for Calabi--Yau varieties with canonical singularities.
We provide an interpretation of our construction using the theory of motivic integration.
\end{abstract}



\setcounter{tocdepth}{1}
\tableofcontents

\section{Introduction}

\subsection{Background: mirror symmetry}

BCOV torsion, introduced by Bershadsky, Cecotti, Ooguri, and Vafa in the outstanding papers \cite{bcov,bcov2},  is a real valued invariant for Calabi--Yau manifolds equipped with Ricci-flat metrics \cite{Yau78}.
More precisely, let $X$ be a Calabi--Yau manifold, i.e., a compact K\"ahler manifold with trivial canonical bundle,
and let $\omega$ be a Ricci-flat K\"ahler metric,
the \emph{BCOV torsion} of $(X,\omega)$ is the weighted product
\begin{equation}
\label{eq-intro-BCOVTorsion}
\mathcal{T}_{\operatorname{BCOV}}(X, \omega):=\prod_{p=1}^{\dim X} \mathcal{T}_p^{(-1)^pp} \;,
\end{equation}
where $\mathcal{T}_p$ is the analytic torsion,
introduced by Ray--Singer \cite{rs2},
of the $p$-th exterior power of the holomorphic cotangent bundle $\bigwedge^{p}(T^*{X})$ equipped with the induced metric.

The motivation of
Bershadsky--Cecotti--Ooguri--Vafa \cite{bcov, bcov2} comes from string theory and has impact on mirror symmetry,
which predicts that for a family of Calabi--Yau manifolds, there is another family of Calabi--Yau manifolds with maximal degeneration, called the \textit{mirror} family,
such that the symplectic geometry (e.g.~Gromov--Witten invariants) of the first family, called the A-model,
is ``equivalent'' to the complex geometry (e.g.~variation of Hodge structures) of the mirror family, called the B-model.
Candelas--de la Ossa--Green--Parkes \cite{cdgp} conjectured a precise relation
between the potential ($J$-function) of the genus zero Gromov--Witten invariants of quintic threefolds (A-model)
and the potential ($I$-function) of the Yukawa coupling for the quintic mirror family (B-model).
Such a relation is expected to hold in general for mirror Calabi--Yau families (see \cite{MorrisonJAMS93}) and gives
surprising predictions in enumerative geometry. This genus zero mirror symmetry conjecture
was proved by Givental \cite{GiventalIMRN, Givental98} and Lian--Liu--Yau \cite{LianLiuYau} for a large class of examples including the original case of quintic threefolds.
Bershadsky--Cecotti--Ooguri--Vafa \cite{bcov,bcov2} computed certain invariants on the B-model that conjecturally correspond to \emph{higher genus} Gromov--Witten invariants.
This allows them to put forth conjectural formulas for all genus Gromov--Witten invariants of quintic threefolds.
The genus one part of this conjecture was proved by Zinger \cite{Zinger08, ZingerJAMS} in the broader setting of Calabi--Yau hypersurfaces in projective spaces. A lot of progress has been made recently on the study of Gromov--Witten invariants of genus $\geqslant 2$
(see \cite{CLLL15, CLLL16, GuoJandaRuan17, GuoJandaRuan18, NMSP1, NMSP2, NMSP3, ChenJandaRuan19} and references therein).
Despite the recent increasing interest on A-model invariants in genus $\geqslant 2$, the research on B-model invariants is currently focused on the genus one theory and
the particular case of Bershadsky--Cecotti--Ooguri--Vafa's B-model invariant corresponding to the genus one Gromov--Witten invariant is the aforementioned BCOV torsion \eqref{eq-intro-BCOVTorsion}.

The central object in this paper is the following normalization of the BCOV torsion,
called the \emph{BCOV invariant}.
Let $X$ be an $n$-dimensional Calabi--Yau manifold equipped with a Ricci-flat metric of K\"ahler form $\omega$,
its BCOV invariant  \cite{fly, efm} is defined\footnote{Here we use the definition of \cite{efm}, which differs from the one in \cite{fly} by an explicit power of $2\pi$.} by
\begin{equation}
\label{eq-intro-T}
\mathcal{T}(X):=\mathcal{T}_{\operatorname{BCOV}}(X, \omega)
\left(\prod_{k=1}^n\operatorname{covol}_{L^2}\left(H^k(X, \Z),\omega\right)^{(-1)^kk}\right)\left(\big(2\pi\big)^{-n}\int_X\frac{\omega^n}{n!}\right)^{\frac{\chi(X)}{12}},
\end{equation}
where $\chi(X)$ is the topological Euler characteristic of $X$,
and $\operatorname{covol}_{L^2}\left(H^k(X, \Z)\right)$ is
the covolume of the lattice $\im\left(H^k(X, \Z)\to  H^k(X, \R)\right)$  with respect to the $L^2$-metric induced by $\omega$.
The virtue of the BCOV invariant is that it depends only on the complex structure, but not on the K\"ahler metric.
Fang--Lu--Yoshikawa \cite{fly}
constructed the BCOV invariant for strict Calabi--Yau threefolds
and studied its asymptotic behavior along degenerations.
Their work confirmed the conjectural formula of Bershadsky--Cecotti--Ooguri--Vafa \cite{bcov,bcov2} for the BCOV invariant near the large complex structure limit of the quintic mirror family (see \cite[Conjecture 1.2 (B)]{fly}).
Eriksson--Freixas i Montplet--Mourougane \cite{efm}
generalized the construction as well as the asymptotic study of the BCOV invariant to Calabi--Yau manifolds of arbitrary dimension. They proved  in \cite{efm2}
a higher-dimensional generalization of the conjectured formula mentioned above,  and showed the compatibility with Zinger's result on the A-model \cite{Zinger08, ZingerJAMS}, thus completing the genus one mirror symmetry conjecture of Bershadsky--Cecotti--Ooguri--Vafa \cite{bcov,bcov2} in this case.

\medskip

Throughout this paper,
for an $n$-dimensional Calabi--Yau manifold $X$,
we will use the following ``normalized logarithmic BCOV invariant'':
\begin{equation}
\label{eq-tau-Tau}
\tau(X) = \log\mathcal{T}(X) + \frac{\log(2\pi)}{2} \sum_{k=0}^{2n} (-1)^kk(k-n)b_k(X) \;,
\end{equation}
where $b_k(X)$ is the $k$-th Betti number of $X$.
This normalization comes from \cite[(0.13)]{z2}.

\subsection{Birational invariance conjecture}
As two birationally isomorphic Calabi--Yau varieties share the same mirror,
their BCOV invariants should coincide.
This leads to the following conjecture.

\begin{conj}
\label{conj-BirInv}
For birational Calabi--Yau manifolds $X$ and $X'$,
we have $\tau(X)=\tau(X')$.
\end{conj}

In view of \eqref{eq-tau-Tau}, Conjecture \ref{conj-BirInv} is equivalent to say that $\mathcal{T}(X)=\mathcal{T}(X')$, since birational Calabi--Yau manifolds have the same Betti numbers, by Batyrev \cite{Batyrev99}.

Conjecture \ref{conj-BirInv} was proposed in dimension 3 by Yoshikawa \cite[Conjecture 2.1]{yo06},
and in arbitrary dimension by Eriksson--Freixas i Montplet--Mourougane \cite[Conjecture B]{efm}.
Fang--Lu--Yoshikawa \cite[Conjecture 4.17]{fly} stated a weaker form of this conjecture.

Since the BCOV invariant can be thought as a ``secondary'' analogue of variation of Hodge structures associated with deformations of Calabi--Yau manifolds,
Conjecture~\ref{conj-BirInv} is a ``secondary'' analogue of the theorem of Batyrev \cite{Batyrev99} and Kontsevich \cite{Kontsevich} that birational Calabi--Yau manifolds have the same Hodge numbers.

Several results were obtained towards Conjecture \ref{conj-BirInv}:
\begin{itemize}
\item Maillot and R{\"o}ssler \cite[Theorem 1.1]{ma-ro} showed that
for two smooth projective Calabi--Yau threefolds $X$, $X'$ defined over a subfield $K$ of $\C$
such that $X_\C$ and $X_\C'$ are birational,
then for any fixed finite set $T$ of complex embeddings of $K$,
there exist $n\in \N_{>0}$ and $\alpha\in K^{\times}$,  such that
\begin{equation}\label{eq-Maillot-Rossler}
\tau(X_\sigma') - \tau(X_\sigma) = \frac{1}{n} \log \big|\sigma(\alpha)\big|
\hspace{5mm} \text{for any } \sigma\in T \;,
\end{equation}
where $X_{\sigma}:=X\otimes_{K, \sigma}\C$.
Maillot and R{\"o}ssler also proved the same result under the strictly more general hypothesis that
$X_\C$ and $X_\C'$ are derived equivalent\footnote{i.e., their bounded derived categories of coherent sheaves are equivalent as $\C$-linear triangulated categories. Note that the derived equivalence of birational Calabi--Yau threefolds was proved by Bridgeland \cite[Theorem 1.1]{brid}, and there are derived equivalent Calabi--Yau threefolds that are not birationally equivalent \cite{BorCal09}, \cite{Caldararu07}, \cite{MR2846680}.}.
\item The second author \cite[Corollary 0.5]{z} proved Conjecture \ref{conj-BirInv} for Atiyah flops of $(-1, -1)$-curves in Calabi--Yau threefolds.
\end{itemize}


\subsection{Main results}
In this paper, we confirm Conjecture \ref{conj-BirInv}.
\begin{alphatheorem}
\label{thm-main-BirInv}
Let $X$ and $X'$ be projective Calabi--Yau manifolds.
If $X$ and $X'$ are birationally isomorphic,
then $\tau(X) = \tau(X')$.
\end{alphatheorem}

The BCOV invariants can be extended to projective manifolds with torsion canonical bundle
(or equivalently, with vanishing first Chern class by \cite{BeauvilleJDG}), see \cite{z2}.
Theorem \ref{thm-main-BirInv} still holds in this more general case.
In fact, we can prove the birational invariance in a much broader setting involving singular varieties.

We call a normal projective complex variety $X$ a \emph{ Calabi--Yau variety with canonical (resp. KLT) singularities},
if it has canonical (resp. Kawamata log terminal) singularities (cf.~\cite{KollarMori} or Definition \ref{def-KLT}) and $K_{X}\sim_{\Q}0$,
where $\sim_\Q$ is the linear equivalence relation for $\Q$-Cartier divisors.
We will propose a natural definition of the BCOV invariant for Calabi--Yau varieties with KLT singularities (see Definition \ref{def-BCOV-KLT}),
which we still denote by $\tau$.
It coincides with the usual one in the smooth case.
Theorem~\ref{thm-main-BirInv} admits the following extension:

\begin{alphatheorem}
\label{thm-main-singular}
Let $X$ and $X'$ be Calabi--Yau varieties with canonical singularities.
If $X$ and $X'$ are birationally isomorphic,
then $\tau(X) = \tau(X')$.
\end{alphatheorem}

A resolution of singularities $f\colon \widetilde{X}\to X$ is called \emph{crepant} if the relative canonical divisor $K_{\widetilde{X}/X}$ is trivial.
By Theorem \ref{thm-main-singular},
the BCOV invariant of a canonical Calabi--Yau variety equals to the BCOV invariant of any crepant resolution.
Note that neither the existence nor the uniqueness of crepant resolution is guaranteed.
Bridgeland--King--Reid \cite{BKR} proved its existence in dimension 3 for Gorenstein quotient singularities.

It is worth mentioning that in the recent work \cite{DaiYoshikawa20}, Dai and Yoshikawa constructed examples showing that certain analytically defined BCOV invariants for orbifolds is \textit{not} a birational invariant already in dimension 2. The orbifold surfaces in their examples have singular points worse than canonical (i.e.~du Val) singularities, hence compatible with our Theorem \ref{thm-main-singular}. Nevertheless, quotient singularities are KLT and it is highly interesting to compare our extended definition and the analytic definition for Calabi--Yau orbifolds; see Remark~\ref{rmk-OrbifoldBCOV}.


\medskip

The \emph{curvature formula} is of fundamental importance in the theory of BCOV invariants.
We refer the readers to \cite[Theorem 4.9]{fly}, \cite[Proposition 5.10]{efm} and \cite[Theorem 0.4]{z2} for the precise formulation in the smooth case.
We have the following curvature formula for the BCOV invariant of locally trivial deformation families
(in the sense of Flenner--Kosarew \cite{FlennerKosarew87}, cf. Definition \ref{def-LocTrivial})
of KLT Calabi--Yau varieties.

\begin{alphatheorem}
\label{thm-main-curvature}
Let $S$ be a complex manifold.
Let $\pi\colon \mathcal{X}\to S$ be a flat family of normal projective KLT Calabi--Yau varieties.
Let $X_s = \pi^{-1}(s)$ for $s\in S$.
Assume that $\pi$ is locally trivial.
Then the following function is $\smooth$,
\begin{align}
\begin{split}
\tau(\mathcal{X}/S) \colon S & \rightarrow \R \\
s & \mapsto \tau(X_s) \;.
\end{split}
\end{align}
Moreover,
we have the following identity of $(1,1)$-forms on $S$,
\begin{equation}
\frac{\overline{\partial}\partial}{2\pi i}\tau(\mathcal{X}/S) =
\omega_{\mathrm{Hdg},\mathcal{X}/S} - \frac{\chi(X)}{12} \omega_{\mathrm{WP},\mathcal{X}/S} \;,
\end{equation}
where $\chi(X)$ is the stringy Euler characteristic of $X_s$ (cf. Definition \ref{def-stringy-inv}),
$\omega_{\mathrm{Hdg},\mathcal{X}/S}$ is the Hodge form of the family $\mathcal{X}/S$ (cf. Definition \ref{def-stringy-Hodge-form})
and $\omega_{\mathrm{WP},\mathcal{X}/S}$ is the Weil--Petersson form of the family $\mathcal{X}/S$ (cf. Definition \ref{def-WP-form}).
\end{alphatheorem}

It is still in active research to lay a rigorous foundation of a mathematical theory of B-model invariants in genus $\geqslant 2$. Once such a theory is built, its birational invariance will be of great importance. We hope that our results on genus one can serve as the first step towards the big picture.

\subsection{Overview of  proof}
To highlight the key ideas, we only explain the proof of Theorem \ref{thm-main-BirInv} here.
The proof contains three main ingredients.

\paragraph{a) \emph{BCOV invariant for pairs}}
The BCOV invariant for Calabi--Yau manifolds was extended by the second author \cite{z2} to all pairs $(X,\gamma)$
with $X$ a compact K{\"a}hler manifold
and $\gamma$ a meromorphic canonical form on $X$
such that $\div(\gamma)$ is of simple normal crossing support and without component of multiplicity $-1$.

We denote $\div(\gamma) = D = m_1D_1+\dots +m_lD_l$ and $D_J=\bigcap_{j\in J}D_j$ for $J\subseteq\{1,\dots,l\}$. By convention, $D_{\emptyset}=X$ and $\prod_{j\in \emptyset} \frac{-m_j}{m_j+1}=1$.
The BCOV invariant of $(X,\gamma)$ is defined by
\begin{equation}
\label{eq-intro-tau-pair}
\tau(X,\gamma) = \sum_{J\subseteq\{1,\dots,l\}} \left(\prod_{j\in J} \frac{-m_j}{m_j+1}\right) \tau_\mathrm{BCOV}(D_J,\omega) + \text{correction terms},
\end{equation}
where $\omega$ is a K{\"a}hler form on $X$,
$\tau_\mathrm{BCOV}(D_J,\omega)$ is the (logarithmic) BCOV torsion of $\big(D_J,\omega\big|_{D_J}\big)$ (see Definition \ref{def-bcov-torsion}),
and the correction terms are given by Bott--Chern forms, making $\tau(X,\gamma)$ independent of $\omega$ (see Definition \ref{def-BCOV-Pair}).

\paragraph{b) \emph{Blow-up formula}}
The second author \cite{z2} worked out the precise behavior of the extended BCOV invariant \eqref{eq-intro-tau-pair} under a blow-up (see Theorem \ref{thm-bl}).
The formula of the second author expresses the change of the BCOV invariant under a blow-up
in terms of the BCOV invariant of projective spaces endowed with some canonical form,
together with certain topological data.
The work of the second author is based on the technique of deformation to the normal cone of Baum--Fulton--MacPherson \cite[\textsection 1.5]{bfm},
and on series of work of Bismut and his collaborators on the Quillen metric \cite{bb, b97, b04, ble}.

\paragraph{c) \emph{Birational BCOV}}
To confirm Conjecture \ref{conj-BirInv},
by the weak factorization theorem of Abramovich, Karu, Matsuki and W{\l}odarczyk \cite{akmw, MR2013783},
it suffices to normalize the BCOV invariant $\tau(X, \gamma)$ in \eqref{eq-intro-tau-pair}
in such a way that the normalized BCOV invariant does not change under blow-ups.
The normalization in this paper is a linear combination of the Betti numbers of the strata $\big\{D_J\big\}_{J\subseteq\{1,\cdots,l\}}$.

\subsection{BCOV invariants and motivic integration}
\label{subsec:BCOVandMI}
It might seem mysterious that the weighted sum in \eqref{eq-intro-tau-pair} happens to be the right object to study,
which eventually allows us to prove the birational invariance of the BCOV invariant.
One conceptual progress made in this paper is an explanation of the construction of $\tau(X,\gamma)$ (see \eqref{eq-intro-tau-pair})
using Kontsevich's motivic integration \cite{Kontsevich}.
It is worth mentioning that the technique of motivic integration already appeared in a less refined form in the paper of Maillot--R\"ossler \cite{ma-ro} mentioned above.

Let $(X,\gamma)$ be as in a).
We temporarily assume that $X$ is projective and $D = \div(\gamma)$ is effective.
Let $Z(X, \mathcal{I}_D;\mathbb{L}^{-1})$ be the motivic Igusa zeta function (see \textsection \ref{subsect:MotInt}) associated with $(X,D)$, evaluated at $\mathbb{L}^{-1}$.
Its Hodge realization can be computed as follows:
\begin{equation}
\label{eq-intro-MotInt}
H^\bullet\big(Z(X,\mathcal{I}_D;\mathbb{L}^{-1})\big)
= \sum_{J\subseteq\{1,\cdots,l\}} \mathbf{L}^{|J|-n} \left( \prod_{j\in J} \frac{1-\mathbf{L}^{m_j}}{\mathbf{L}^{m_j+1}-1} \right) H^\bullet(D_J) \;,
\end{equation}
where $\mathbf{L}$  on the right hand side is understood as the operator of tensoring with the Lefschetz Hodge structure $\Z(-1)$.
The following observation is crucial,
\begin{equation}
\label{eq0-intro-MotInt}
H^\bullet\big(Z(X,\mathcal{I}_D;\mathbb{L}^{-1})\big)\big|_{\mathbf{L}=1}
= \sum_{J\subseteq\{1,\cdots,l\}} \left(\prod_{j\in J} \frac{-m_j}{m_j+1}\right) H^\bullet(D_J) \;,
\end{equation}
where the coefficients are exactly the same as in \eqref{eq-intro-tau-pair}.
Using \eqref{eq0-intro-MotInt},
we will show in \S \ref{ch-motivic} that the BCOV invariant $\tau(X,\gamma)$ is essentially the Quillen metric on
\begin{equation}
\label{eq-intro-MotInt-det}
\bigotimes_k \Big(\det H^k\big(Z(X,\mathcal{I}_D;\mathbb{L}^{-1})\big)\Big)^{(-1)^kk} \;.
\end{equation}

By the change of variables formula in motivic integration (cf. Theorem \ref{thm-MotInt}), the virtual Hodge structure
$H^\bullet\big(Z(X,\mathcal{I}_D;\mathbb{L}^{-1})\big)$ in \eqref{eq-intro-MotInt} is a birational invariant.
Hence the virtual determinant line in \eqref{eq-intro-MotInt-det} is also a birational invariant,
in the sense that there is a canonical isomorphism between the complex lines \eqref{eq-intro-MotInt-det} associated with two birational varieties equipped with effective simple normal crossing divisors.
This partially explains the reason why $\tau(X, \gamma)$ is almost a birational invariant.

\medskip
This paper is organized as follows.

In \textsection \ref{ch-pr},
we give a reminder on the Quillen metric and the BCOV torsion.
A discussion on simple normal crossing divisors is also included.

In \textsection \ref{ch-log},
we develop some basic properties of the so-called localizable and log-type invariants, which will appear repeatedly throughout the paper.

In \textsection \ref{ch-bcov},
we recall the construction of $\tau(X,\gamma)$
and collect several fundamental properties of the BCOV invariant.

In \textsection \ref{ch-motivic},
we explain the construction of $\tau(X,\gamma)$ using motivic integration.

In \textsection \ref{ch-birat},
we construct a birational BCOV invariant.

In \textsection \ref{ch-singular},
we extend the BCOV invariant to the singular cases and prove Theorem \ref{thm-main-curvature}.

In \textsection \ref{ch-proof},
we prove Theorem \ref{thm-main-BirInv} and Theorem \ref{thm-main-singular}.

\medskip

\noindent \textbf{Convention:}
When we write a divisor $D=\sum_{j=1}^l m_j D_j$,
we implicitly assume that the $D_j$'s are distinct prime divisors.
For a complex manifold $X$ and a complex submanifold $Y$,
we denote by $\Bl_YX$ the blow-up of $X$ along $Y$.

\medskip

\noindent \textbf{Acknowledgement:}
We are indebted to Professor Ken-Ichi Yoshikawa for his guidance throughout the project.
We would like to thank
Professor Xianzhe Dai,
Professor Dennis Eriksson,
Professor Gerard Freixas i Montplet,
Professor Chen Jiang,
Professor Xiaonan Ma,
and Professor Vincent Maillot
for helpful discussions.

Lie Fu is supported by the Radboud Excellence Initiative from the Radboud University, by the project FanoHK (ANR-20-CE40-0023) of Agence Nationale de la Recherche in France, and by the University of Strasbourg Institute for Advanced Study (USIAS).

Yeping Zhang is supported by KIAS individual Grant MG077401 at Korea Institute for Advanced Study.

\section{Preliminaries}
\label{ch-pr}

\subsection{Quillen metric and topological torsion}
\label{subsec-top-torsion}

Let $X$ be a compact K{\"a}hler manifold of dimension $n$.
For any holomorphic vector bundle $E$ over $X$,
its \textit{determinant} line of cohomology \cite{KnudsenMumford} is
\begin{equation}
\lambda(E)=\det H^\bullet(X, E):=\bigotimes_{q=0}^n \left(\det H^q(X, E)\right)^{(-1)^q}.
\end{equation}
For any  K\"ahler metric on $X$ and any Hermitian metric on $E$,
one can define the so-called \textit{Quillen metric} \cite{q} on the determinant line $\lambda(E)$,
see \cite[Definition 1.10]{ble}.

For $p=0,\cdots,n$,
set
\begin{equation}
\label{eq-def-lambda-p}
\lambda_p(X) = \lambda\left(\bigwedge^p(T^*X)\right)=\bigotimes_{q=0}^n \Big(\det H^{p,q}(X)\Big)^{(-1)^q} \;.
\end{equation}
The determinant lines $\lambda_p(X)$, with their Quillen metrics, will be the basic building blocs in the construction of the BCOV invariant. Before doing that, let us first introduce a natural invariant that will appear later, called topological torsion, and recall its vanishing.

Set
\begin{equation}
\label{eq-def-eta-dR}
\eta(X) =\det H_{\mathrm{dR}}^\bullet(X):= \bigotimes_{k=0}^{2n} \Big(\det H^k_\mathrm{dR}(X)\Big)^{(-1)^k} \;.
\end{equation}
By the Hodge decomposition
\begin{equation}
\label{eq-Hodge-dec}
H^k_{\mathrm{dR}}(X)=\bigoplus_{p+q=k}H^{p,q}(X), \text{ for any } 0\leqslant k\leqslant n,
\end{equation}
we have
\begin{equation}
\label{eq-def-eta}
\eta(X) = \bigotimes_{p=0}^n \Big(\lambda_p(X)\Big)^{(-1)^p} \;.
\end{equation}

We fix a square root of $i$.
This choice will be irrelevant.
We identify the de Rham cohomology $H^k_\mathrm{dR}(X)$ with the singular cohomology $H^k_\mathrm{Sing}(X,\C)$ as follows,
\begin{align}
\label{eq-dR-Sing}
\begin{split}
H^k_\mathrm{dR}(X) & \rightarrow H^k_\mathrm{Sing}(X,\C) \\
[\alpha] & \mapsto \Big[\mathfrak{a} \mapsto \big(2\pi i\big)^{-k/2} \int_\mathfrak{a}\alpha\Big]\;,
\end{split}
\end{align}
where $\alpha$ is a closed $k$-form
and $\mathfrak{a}$ is a $k$-chain in $X$.
The identification \eqref{eq-dR-Sing} endows $H^k_{\mathrm{dR}}(X)$ with an integral structure.
Let  $\epsilon_X$ be a generator of the induced integral structure on $\eta(X)$.
More precisely,
for $k=0,\cdots,2n$,
let
\begin{equation}
\label{eq-def-basis-cohomology}
\sigma_{k,1},\cdots,\sigma_{k,b_k}
\in H^k_\mathrm{Sing}(X,\Z)_{\operatorname{tf}}
\end{equation}
be a $\Z$-basis of the quotient of $H^k_\mathrm{Sing}(X,\Z)$ modulo its subgroup of torsion elements.
Then $\sigma_{k,1}, \cdots, \sigma_{k,b_k}$ form a basis of $H^k_\mathrm{dR}(X)$.
Set
\begin{align}
\label{eq-def-epsilon}
\begin{split}
\epsilon_X = \bigotimes_{k=0}^{2n} \big(\sigma_{k,1}\wedge\cdots\wedge\sigma_{k,b_k}\big)^{(-1)^k} \in \eta(X) \;,
\end{split}
\end{align}
which is well-defined up to $\pm 1$.

Let $\omega$ be a K{\"a}hler form on $X$, which induces a Hermitian metric on $\bigwedge^p(T^*X)$ for any $p$.
Let $\big\lVert\cdot\big\rVert_{\lambda_p(X),\omega}$
be the Quillen metric on $\lambda_p(X)$ associated with $\omega$.
Let $\big\lVert\cdot\big\rVert_{\eta(X)}$ be the metric on $\eta(X)$
induced by $\big\lVert\cdot\big\rVert_{\lambda_p(X),\omega}$ via \eqref{eq-def-eta}. Proceeding in the same way as in the proof of \cite[Theorem 2.1]{z},
we can show that $\big\lVert\cdot\big\rVert_{\eta(X)}$ is independent of $\omega$.

\begin{defn}
\label{def-tau-top}
We define the \emph{topological torsion} of $X$ as
\begin{equation}
\label{eq-def-tau-top}
\tau_\mathrm{top}(X) = \log \big\lVert\epsilon_X\big\rVert_{\eta(X)} \;.
\end{equation}
\end{defn}

The identification \eqref{eq-dR-Sing} allows us to have the following vanishing result.

\begin{prop}[{\cite[Proposition 1.24]{z2}}]
For any compact K\"ahler manifold $X$, we have
\begin{equation}
\label{eq-prop-rei}
\tau_\mathrm{top}(X) = 0 \;.
\end{equation}
\end{prop}

\subsection{BCOV torsion} Keep the same setting of \S \ref{subsec-top-torsion}.
Following \cite[\S 5.8]{bcov2}, we consider the weighted product of determinant lines $\lambda_p(X)$ defined in \eqref{eq-def-lambda-p}.
\begin{equation}
\label{eq-def-lambda}
\lambda(X)
= \bigotimes_{0\leqslant p,q\leqslant n} \Big( \det H^{p,q}(X) \Big)^{(-1)^{p+q}p}
= \bigotimes_{p=1}^n \Big(\lambda_p(X)\Big)^{(-1)^pp} \;.
\end{equation}
Set
\begin{equation}
\label{eq-def-lambda-dr}
 \lambda_\mathrm{dR}(X)
 = \bigotimes_{k=1}^{2n} \Big(\det H^k_\mathrm{dR}(X)\Big)^{(-1)^kk}\;.
\end{equation}
By the Hodge decomposition \eqref{eq-Hodge-dec}, we have
\begin{equation}
\label{eq-def-lambda-dR}
\lambda_\mathrm{dR}(X) = \lambda(X) \otimes \overline{\lambda(X)}\;.
\end{equation}
The identity \eqref{eq-def-lambda-dR} appeared in Kato \cite[last identity in \textsection 1.3]{Kato14} and was first applied to this setting in \cite{efm}.

Let $\big\lVert\cdot\big\rVert_{\lambda(X),\omega}$
be the metric on $\lambda(X)$ induced by $\big\lVert\cdot\big\rVert_{\lambda_p(X),\omega}$ via \eqref{eq-def-lambda}.
Let $\big\lVert\cdot\big\rVert_{\lambda_\mathrm{dR}(X),\omega}$
be the metric on $\lambda_\mathrm{dR}(X)$ induced by $\big\lVert\cdot\big\rVert_{\lambda(X),\omega}$ via \eqref{eq-def-lambda-dR}.
Let $\sigma_X$ be the integral generator of $\lambda_{\mathrm{dR}}(X)$ defined as follows,
using the $\Z$-basis of $H^\bullet_{\mathrm{Sing}}(X,\Z)_{\operatorname{tf}}$ in \eqref{eq-def-basis-cohomology},
\begin{equation}
	\sigma_X = \bigotimes_{k=1}^{2n} \big(\sigma_{k,1}\wedge\cdots\wedge\sigma_{k,b_k}\big)^{(-1)^kk}.
\end{equation}

\begin{defn}
\label{def-bcov-torsion}
We define the \emph{BCOV torsion} of $(X, \omega)$ as
\begin{equation}
\label{eq-def-bcov-torsion}
\tau_\mathrm{BCOV}(X,\omega) =
\log \big\lVert\sigma_X\big\rVert_{\lambda_\mathrm{dR}(X),\omega} \;.
\end{equation}
\end{defn}

In the case where $X$ is a Calabi--Yau manifold equipped with a Ricci-flat metric $\omega$,
this invariant $\tau_{\operatorname{BCOV}}(X, \omega)$ is precisely the logarithm of
the product of the first two factors on the right hand side of \eqref{eq-intro-T}.

\subsection{Divisor with simple normal crossing support}
\label{subsect-snc}

For $I\subseteq\big\{1,\cdots,n\big\}$,
set
\begin{equation}
\C^n_I =
\Big\{ (z_1,\cdots,z_n)\in \C^n \;:\; z_i = 0 \hspace{2.5mm} \text{for } i\in I \Big\}
\subseteq \C^n \;.
\end{equation}

Let $X$ be a complex manifold of dimension $n$.
Let $Y_1,\cdots,Y_l \subseteq X$ be closed complex submanifolds.

\begin{defn}
\label{def-ti}
We say that $Y_1,\cdots,Y_l$ \textit{transversally intersect} if
for any $x\in X$,
there exists a holomorphic local chart $\C^n \supseteq U \xrightarrow{\varphi} X$
such that
\begin{itemize}
\item[-] $0\in U$ and $\varphi(0) = x$;
\item[-] for each $k=1,\cdots,l$,
either $\varphi^{-1}(Y_k) = \emptyset$, \\
or $\varphi^{-1}(Y_k) = U \cap \C^n_{I_k}$ for certain $I_k\subseteq\big\{1,\cdots,n\big\}$.
\end{itemize}
\end{defn}

\begin{rem}
\textit{The definition of transversal intersection above is more general than the usual one.}
Let $Y,Z \subseteq X$ be connected complex submanifold which transversally intersect in the sense of Definition \ref{def-ti}.
If $Y \subseteq X$ is of codimension $1$,
then
\begin{itemize}
\item[-] either $Y$ and $Z$ transversally intersect in the usual sense,
\item[-] or $Z \subseteq Y$.
\end{itemize}
\end{rem}

We denote
\begin{equation}
D = \sum_{j=1}^l m_j D_j \;,
\end{equation}
where $m_j\in\Z\backslash\{0\}$
and $D_1,\cdots,D_l \subseteq X$ are mutually distinct prime divisors.

\begin{defn}
\label{def-snc}
A divisor $D$ on $X$ is
said \textit{with simple normal crossing support}
if $D_1,\cdots,D_l$ are smooth and transversally intersect.
\end{defn}

Now let $D$ be a divisor on $X$ with simple normal crossing support.
Let $L$ be the holomorphic line bundle $\mathcal{O}_X(D)$.
Denote by $\mathscr{M}(X,L)$ the space of meromorphic sections of $L$.
Let $\gamma\in\mathscr{M}(X,L)$ such that $\div(\gamma) = D$.
Let $L_j$ be the normal line bundle of $D_j \hookrightarrow X$.

\begin{defn}
\label{def-res}
We define $\mathrm{Res}_{D_j}(\gamma)\in\mathscr{M}(D_j,L \otimes L_j^{-m_j})$ as the image of $\gamma\in\mathscr{M}(X,L)$ via the canonical isomorphism
\begin{equation}
L(-m_jD_j)\big|_{D_j} \simeq L\big|_{D_j} \otimes L_j^{-m_j} \;.
\end{equation}
\end{defn}

Let $\C^n \supseteq U \xrightarrow{\varphi} X$ be a local chart as in Definition \ref{def-ti}.
Assume that
\begin{equation}
\gamma\big|_{\varphi(U)} = s \varphi_*\big(z_1^{m_1} \cdots z_n^{m_r}\big) \;,
\end{equation}
where $0 \leqslant r \leqslant n$ and $s\in H^0(\varphi(U),L)$ is nowhere vanishing.
For $j=1,\cdots,r$,
we have
\begin{equation}
\mathrm{Res}_{D_j}(\gamma) \big|_{D_j \cap \varphi(U)} = s \varphi_*\big(z_1^{m_1} \cdots z_{j-1}^{m_{j-1}} z_{j+1}^{m_{j+1}}\cdots z_r^{m_r} (dz_j)^{m_j} \big) \;,
\end{equation}
where $dz_j$ is viewed as a conormal vector of $\{z_j=0\} \subseteq \C^n$.

Note that
\begin{equation}
\div\big(\mathrm{Res}_{D_1}(\gamma)\big) = \sum_{j=2}^l m_j \big( D_1 \cap D_j \big) \;.
\end{equation}
The following identity holds in $ \mathscr{M}\left(D_1 \cap D_2,L \otimes L_1^{-m_1} \otimes L_2^{-m_2}\right)$,
\begin{align}
\label{eq-res-commute}
\begin{split}
\mathrm{Res}_{D_1 \cap D_2}\big(\mathrm{Res}_{D_1}(\gamma)\big)
= \mathrm{Res}_{D_1 \cap D_2}\big(\mathrm{Res}_{D_2}(\gamma)\big)  \;.
\end{split}
\end{align}
In other words,
the order of taking $\mathrm{Res}_\cdot(\cdot)$ does not matter.

\section{Localizable invariants}
\label{ch-log}

\subsection{Definitions and examples}


\begin{defn}
\label{def0-LocInv}
Let $\Kah$ be the category of compact K\"ahler manifolds. Let $\phi: \Kah \rightarrow \R$ be a function that depends only on the isomorphism classes of compact K\"ahler manifolds.
\begin{itemize}
\item $\phi$ is called a \emph{localizable invariant} if for any $X, X'\in\Kah$, and closed complex submanifolds $Y\subseteq X$, $Y'\subseteq X'$ such that $Y \simeq Y'$ and $N_{Y/X} \simeq N_{Y'/X'}$, we have
\begin{equation}
\label{eq-bl-relation}
\phi\big(\Bl_YX\big) - \phi(X) = \phi\big(\Bl_{Y'}X'\big) - \phi(X') \;.
\end{equation}
\item $\phi$ is called of \emph{log-type}\footnote{The terminology refers to the fact that if $\chi(X)\neq 0$, then
$\frac{\phi(\mathbb{P}(V))}{\chi(\mathbb{P}(V))}=\frac{\phi(X)}{\chi(X)}+\frac{\phi(\CP^{r-1})}{\chi(\CP^{r-1})}$.} if
for any $X\in\Kah$ and  $V$ a holomorphic vector bundle of rank $r$ over $X$, we have
\begin{equation}
\label{eq2-twist-relation}
\phi\big(\mathbb{P}(V)\big) = \chi\big(\CP^{r-1}\big)\phi(X) + \chi(X)\phi\big(\CP^{r-1}\big) \;.
\end{equation}
\item $\phi$ is called \emph{additive} if
for any $X\in\Kah$ and $Y\subseteq X$ a closed complex submanifold, we have
\begin{equation}
\label{eq-add-relation}
\phi\big(\Bl_YX\big) - \phi(X) = \phi\big(\mathbb{P}(N_{Y/X})\big) - \phi(Y) \;.
\end{equation}
\end{itemize}
An additive invariant is clearly localizable.
A linear combination of localizable (resp. log-type, additive) invariants is again localizable (resp.~of log-type, additive).
\end{defn}

Let us give several examples of such invariants that will play important roles later.
\begin{ex}
\label{ex-localizable}
Let $X$ be a compact K\"ahler manifold.
\begin{itemize}
\item For any $k\in \N$,
the $k$-th Betti number $b_{k}(X)$ is an additive invariant.
Let
\begin{equation}
P_t(X) = \sum_{k=0}^{2 \dim X} b_k(X) t^k
\end{equation}
be the Poincar{\'e} polynomial.
For any $t\in\R$,
$P_t(X)$ is an additive invariant.
In particular,
the topological Euler characteristic
\begin{equation}
\label{eq-Def-Chi}
\chi(X) = P_{-1}(X)
\end{equation}
is additive.
\item The invariant
\begin{equation}
\label{eq-Def-Chi'}
\chi'(X) = \frac{\d}{\d t}P_t(X) \Big|_{t=-1} = \dim(X) \chi(X)
\end{equation}
is of log-type and additive.
To show that $\chi'(X)$ is of log-type (i.e., identity \eqref{eq2-twist-relation}),
we take the derivative of the identity
$P_t(\mathbb{P}(V)) = P_t(X)P _t(\CP^{r-1})$.
\item The invariant
\begin{equation}
\label{eq-Def-Chi''}
\chi''(X) = \frac{\d^2}{\d t^2}P_t(X) \Big|_{t=-1} - \dim(X)^2 \chi(X)
= P_t(X) \frac{\d^2}{\d t^2} \log P_t(X) \Big|_{t=-1}
\end{equation}
is of log-type and localizable (but not additive).
To show that $\chi''(X)$ is of log-type (i.e., identity \eqref{eq2-twist-relation}),
we take the second derivative of the logarithm of the identity
$P_t(\mathbb{P}(V)) = P_t(X)P _t(\CP^{r-1})$.
\end{itemize}
\end{ex}

\subsection{Localizable invariant for pairs}
\label{subch-log-pair}

Let $d$ be a non-zero integer.
\begin{defn}
	\label{def-ConditionStard}
	For a compact K\"ahler manifold $X$ and a divisor
	\begin{equation}
	D=\sum_{j=1}^l m_jD_j
	\end{equation}
	on $X$,
	we say that $(X,D)$ satisfies \textbf{condition $(\star_d)$}
	if $D$ is of simple normal crossing support and $m_j\neq -d$ for all $j$.
\end{defn}

We will always use the following notation.
For $J\subseteq\{1,\dots,l\}$, set
\begin{equation}
\label{eq-def-wJ}
w^J_d = \prod_{j\in J} \frac{-m_j}{m_j+d} \;,\hspace{5mm}
D_J =  \bigcap_{j\in J} D_j \;.
\end{equation}
By convention,
$w^d_\emptyset = 1$ and $D_\emptyset = X$.

\begin{defn}
\label{def-LocInv}
Let $\phi$ be a localizable invariant.
Let $d\in \Z\backslash\{0\}$.
For $(X,D)$ satisfying the condition $(\star_d)$,
we define
\begin{equation}
\label{eq-def-LocInv}
\phi_d(X, D) = \sum_{J\subseteq\{1,\dots,l\}} w^J_d \phi(D_J) \;.
\end{equation}
If there is a meromorphic section $\gamma$ of a holomorphic line bundle over $X$
such that  $\div(\gamma) = D$,
we define $\phi_d(X,\gamma) = \phi_{d}(X,D)$.
\end{defn}

Let $[\xi_0:\cdots:\xi_n]\in\CP^n$ be homogenous coordinates.
For $j=0,\cdots,n$,
we denote $H_j = \{\xi_j=0\} \subseteq\CP^n$.
For $m_0,\dots,m_n\in\Z$,
we denote
\begin{equation}
D_{m_0,\cdots,m_n} = \sum_{j=0}^n m_jH_j \;.
\end{equation}

Recall that $\chi$ is the topological Euler characteristic.
Replacing $\phi$ by $\chi$ in Definition \ref{def-LocInv},
we get $\chi_d(\cdot,\cdot)$.

\begin{lemme}
\label{lem-ProjSp}
For $d\in \Z\backslash\{0\}$ and $m_0,\dots,m_n\in\Z\backslash\{-d\}$,
we have
\begin{equation}
\label{eq-lem-ProjSp}
\chi_d\big(\CP^n,D_{m_0,\cdots,m_n}\big)
= \bigg(\prod_{j=0}^n(m_j+d)\bigg)^{-1} d^n \sum_{j=0}^n(m_j+d) \;.
\end{equation}
It follows that
$\chi_d\big(\CP^n,D_{m_0,\cdots,m_n}\big)$ vanishes if and only if $D_{m_0,\cdots,m_n}$ is a $d$-canonical divisor.
\end{lemme}
\begin{proof}
Let $w^J_d$ be as in \eqref{eq-def-wJ}.
By Definition \ref{def-LocInv},
we have
\begin{equation}
\label{eq1-pf-lem-ProjSp}
\chi_d\big(\CP^n,D_{m_0,\cdots,m_n}\big) = \sum_{J\subseteq\{0,\dots,n\}} w^J_d (n+1-|J|) \;.
\end{equation}
Set
\begin{equation}
\label{eq2-pf-lem-ProjSp}
f(t) = \prod_{j=0}^n \Big(t - \frac{m_j}{m_j+d}\Big) = \sum_{J\subseteq\{0,\dots,n\}} w^J_d t^{n+1-|J|} \;.
\end{equation}
By \eqref{eq1-pf-lem-ProjSp} and \eqref{eq2-pf-lem-ProjSp},
we have
\begin{equation}
\label{eq3-pf-lem-ProjSp}
\chi_d\big(\CP^n,D_{m_0,\cdots,m_n}\big) = f'(1) \;,\hspace{5mm}
\frac{f'(1)}{f(1)} = \frac{\d}{\d t}\log f(t)\Big|_{t=1} = \sum_{j=0}^n \frac{m_j+d}{d} \;.
\end{equation}
From \eqref{eq2-pf-lem-ProjSp} and \eqref{eq3-pf-lem-ProjSp},
we obtain \eqref{eq-lem-ProjSp}.
This completes the proof.
\end{proof}


Let $Y$ be a compact K\"ahler manifold and $V$ be a holomorphic vector bundle of rank $r$ over $Y$.
Set $X = \mathbb{P}(V)$.
Let $\pi: X \rightarrow Y$ be the canonical projection.
Let $D$ be a divisor on $X$.
We assume that
there exist a divisor $D_Y$ on $Y$,
non-zero integers $m_1,\cdots,m_s$,
and holomorphic sub-bundles $V_1,\cdots,V_s \subseteq V$ of rank $r-1$, such that
\begin{equation}
D = \pi^*D_Y + \sum_{j=1}^s m_j \mathbb{P}(V_j) \;.
\end{equation}
We further assume that $V_1,\cdots,V_s \subseteq V$ transversally intersect.
In particular,  $s\leqslant r$.
We will use the convention $m_{s+1}=\cdots=m_{r}=0$.
For $y\in Y$,
we denote $Z_y = \pi^{-1}(y)$.
Set
\begin{equation}
D_{Z_y} = \sum_{j=1}^s m_s \big( \mathbb{P}(V_j) \cap Z_y \big) \;.
\end{equation}
Then $\big(Z_y,D_{Z_y}\big)$ is isomorphic to $\big(\CP^{r-1},D_{m_1,\cdots,m_r}\big)$ for any $y\in Y$.
In the sequel,
we omit the index $y$ in $\big(Z_y,D_{Z_y}\big)$.
Such a pair $(X,D)$ will be called a \textit{fibration over $(Y,D_Y)$ with fiber $(Z,D_Z)$}.

\begin{lemme}
\label{lem-ProjBun}
Assume that $D$ satisfies the condition $(\star_{d})$ in Definition \ref{def-ConditionStard}.
We have
\begin{equation}
\chi_d(X,D) = \chi_d(Y,D_Y)\chi_d(Z,D_Z) \;.
\end{equation}
\end{lemme}
\begin{proof}
It is a straightforward computation from Definition \ref{def-LocInv}
by using the fact that $\chi(\cdot)$ is an additive invariant and is multiplicative with respect to  products of varieties.
\end{proof}

\begin{prop}
\label{prop-ProjBun}
Let $\phi$ be a log-type localizable invariant and $d\in \Z\backslash\{0\}$.
Assume that $D$ is a $d$-canonical divisor and satisfies the condition $(\star_{d})$.
We have
\begin{equation}
\label{eq-prop-ProjBun}
\phi_d(X, D) = \chi_d(Y,D_Y)\phi_d(Z,D_Z) \;.
\end{equation}
\end{prop}
\begin{proof}
By Definitions \ref{def0-LocInv} and \ref{def-LocInv},
we have
\begin{equation}
\label{eq1-pf-prop-ProjBun}
\phi_d(X, D) = \chi_d(Y,D_Y)\phi_d(Z,D_Z) + \phi_d(Y,D_Y)\chi_d(Z,D_Z) \;.
\end{equation}
Since $(X,D)$ is $d$-canonical,
so is $(Z,D_Z)$.
By Lemma \ref{lem-ProjSp},
we have
\begin{equation}
\label{eq2-pf-prop-ProjBun}
\chi_d(Z,D_Z) = 0 \;.
\end{equation}
From \eqref{eq1-pf-prop-ProjBun} and \eqref{eq2-pf-prop-ProjBun},
we obtain \eqref{eq-prop-ProjBun}.
This completes the proof.
\end{proof}

For $r\in\N\backslash\{0\}$ and $m_1,\cdots,m_s\in \Z$ with $s\leqslant r$,
let $\big(\CP^r,D_{m_1,\cdots,m_s}\big)$ be such that
\begin{equation}
D_{m_1,\cdots,m_s} = \sum_{j=1}^s m_jH_j \;.
\end{equation}
For $r\in\N\backslash\{0\}$ and $m_1,\cdots,m_s\in \Z$ with $s\leqslant r$,
consider a pair $\big(\CP^r,D_{d;m_1,\cdots,m_s}\big)$ with
\begin{equation}
D_{d;m_1,\cdots,m_s} = - (m_1+\cdots+m_s+rd+d) H_0 + \sum_{j=1}^s m_jH_j \;.
\end{equation}
We remark that $D_{d;m_1,\cdots,m_s}$ is a $d$-canonical divisor.

Let $X$ be a compact K\"ahler manifold.
Let
\begin{equation}
D = \sum_{j=1}^l m_j D_j
\end{equation}
be a divisor on $X$ with simple normal crossing support.
Let $Y\subseteq X$ be a connected complex submanifold of codimension $r$ intersecting $D_1,\cdots,D_l$ transversally (see Definition \ref{def-ti})
and
\begin{equation}
Y \subseteq D_j \hspace{2.5mm}\text{for } j=1,\cdots,s \;;\hspace{5mm}
Y \nsubseteq D_j \hspace{2.5mm}\text{for } j=s+1,\cdots,l \;.
\end{equation}
In particular, $s\leqslant r$.
Set
\begin{equation}
D_Y = \sum_{j=s+1}^l m_j\big(D_j \cap Y\big) \;.
\end{equation}
Let $f: X'\to X$ be the blow-up along $Y$.
Let $\widetilde{D}$ be the strict transformation of $D$.
Let $E = f^{-1}(Y)\subseteq X$ be the exceptional divisor.
Set
\begin{equation}
D' = \widetilde{D} + m_e E \;,\quad\text{ where }
m_e = (r-1)d + m_1 + \cdots + m_s \;.
\end{equation}
Note that if $D$ is a $d$-canonical divisor, then so is $D'$. Indeed,
\begin{align*}
	dK_{X'}&=f^*K_X+d(r-1)E\\
	&=f^*\left(\sum_{j=1}^l m_jD_j\right)+d(r-1)E\\
	&=\sum_{j=1}^{s}m_j(\widetilde{D_j}+E)+\sum_{j=s+1}^lm_j\widetilde{D_j}+d(r-1)E\\
	&=\widetilde{D}+m_e E.
\end{align*}

\begin{prop}
\label{prop-Blowup}
Let $\phi$ be a log-type localizable invariant and $d\in \Z\backslash\{0\}$.
Assume that $D$ is a $d$-canonical divisor and satisfies the condition $(\star_{d})$ in Definition \ref{def-ConditionStard}.
We have
\begin{align}
\label{eq-prop-Blowup}
\begin{split}
& \chi_d(X', D')-\chi_d(X,D) = 0 \;,\\
& \phi_d(X', D')-\phi_d(X,D) \\
& = \chi_d(Y,D_Y)
\Big( \chi_d\big(\CP^{r-1},D_{m_1,\cdots,m_s}\big) \phi_d\big(\CP^1, D_{d;m_e}\big)
- \phi_d\big(\CP^r,D_{d;m_1,\cdots,m_s}\big) \Big) \;.
\end{split}
\end{align}
\end{prop}
\begin{proof}
Denote by $\mathds{1}$ the trivial line bundle.
Set $W=\mathbb{P}(N_{Y/X}\oplus\mathds{1})$.
Let $\pi: W\to Y$ be the canonical projection.
Let $\iota: Y\hookrightarrow W$ be the inclusion by the zero section of $N_{Y/X}$.
Let $g: W' \to W$ be the blow-up along $\iota(Y)$.
Set
\begin{equation}
D_W = \pi^*(D_Y) - (m_e+2d)\mathbb{P}\big(N_{Y/X}\big) + \sum_{j=1}^s m_j \mathbb{P}\big(N_{Y/D_j}\oplus\mathds{1}\big) \;.
\end{equation}
Let $\widetilde{D}_W$ be the strict transformation of $D_W$.
We still use $E$ to denote the exceptional divisor of $g: W' \to W$.
Set $D_{W'} = \widetilde{D}_W + m_e E$.
By Definition \ref{def-LocInv} and \eqref{eq-bl-relation},
we have
\begin{align}
\label{eq1-pf-prop-Blowup}
\begin{split}
\chi_d(X',D')-\chi_d(X,D) & = \chi_d(W',D_{W'})-\chi_d(W, D_W) \;,\\
\phi_d(X',D')-\phi_d(X,D) & = \phi_d(W',D_{W'})-\phi_d(W, D_W) \;.
\end{split}
\end{align}
Note that $(W, D_W)$ is a fibration over $(Y, D_Y)$ with fiber $\big(\CP^r,D_{d;m_1,\cdots,m_s}\big)$ (in the sense explained before Lemma \ref{lem-ProjBun}),
by Lemma \ref{lem-ProjSp}, \ref{lem-ProjBun} and Proposition \ref{prop-ProjBun},
we have
\begin{align}
\label{eq2-pf-prop-Blowup}
\begin{split}
\chi_d(W, D_W) & = \chi_d(Y, D_Y) \chi_d\big(\CP^r,D_{d;m_1,\cdots,m_s}\big) = 0 \;,\\
\phi_d(W, D_W) & = \chi_d(Y, D_Y) \phi_d\big(\CP^r,D_{d;m_1,\cdots,m_s}\big) \;.
\end{split}
\end{align}
We denote $D_E = \widetilde{D}_W\big|_E$.
Note that $W'$ is a fibration over $Y$ with fiber $\Bl_0\CP^r$,
and $\Bl_0\CP^r$ is a fibration over $\CP^{r-1}$ with fiber $\CP^1$,
we can show that $(W', D_{W'})$ is fibration over $(E,D_E)$ with fiber $\big(\CP^1,D_{d;m_e}\big)$.
By  Lemma \ref{lem-ProjSp}, \ref{lem-ProjBun} and Proposition \ref{prop-ProjBun},
we have
\begin{align}
\label{eq3-pf-prop-Blowup}
\begin{split}
\chi_d(W', D_{W'}) & = \chi_d(E,D_E) \chi_d\big(\CP^1,D_{d;m_e}\big) = 0 \;,\\
\phi_d(W', D_{W'}) & = \chi_d(E,D_E) \phi_d\big(\CP^1,D_{d;m_e}\big) \;.
\end{split}
\end{align}
Note that $(E, D_E)$ is fibration over $(Y,D_Y)$ with fiber $\big(\CP^{r-1},D_{m_1,\cdots,m_s}\big)$,
by Lemma \ref{lem-ProjBun},
we have
\begin{equation}
\label{eq4-pf-prop-Blowup}
\chi_d(E, D_E) = \chi_d(Y,D_Y) \chi_d\big(\CP^{r-1},D_{m_1,\cdots,m_s}\big) \;.
\end{equation}
From \eqref{eq1-pf-prop-Blowup}-\eqref{eq4-pf-prop-Blowup},
we obtain \eqref{eq-prop-Blowup}.
This completes the proof.
\end{proof}

\section{BCOV invariants for pairs}
\label{ch-bcov}

In this section, we recall the construction of BCOV invariants, as well as their properties, such as their behavior under projective bundles and blow-ups. See \cite{z2} for more details. Moreover, in \S \ref{subsec:P1andP2}, we compute the BCOV invariants for projective spaces in dimension 1 and 2.

\subsection{BCOV invariants}

Let $X$ be a compact K{\"a}hler manifold.
Let $K_X$ be the canonical bundle of $X$.
Let $d\in\Z\backslash\{0\}$.
Let $\gamma\in\mathscr{M}(X,K_X^d)$ be an invertible element (in the commutative ring $\bigoplus_{k\in\Z}\mathscr{M}(X,K_X^k)$).
In other words,
$\gamma$ is non-zero on any connected component of $X$.
We denote
\begin{equation}
\div(\gamma) = D = \sum_{j=1}^l m_j D_j \;,
\end{equation}
where $m_j \in \Z\backslash\{0\}$
and $D_1,\cdots,D_l \subseteq X$ are mutually distinct prime divisors.

\begin{defn}
\label{def-cy-pair}
We call $(X,\gamma)$ a $d$-\emph{Calabi--Yau pair}
if $(X,\mathrm{div}(\gamma))$ satisfies the condition $(\star_d)$ in Definition \ref{def-ConditionStard}.
In particular,
if $X$ is Calabi--Yau and $\gamma \in H^0(X,K_X)$ is nowhere vanishing,
then $(X,\gamma)$ is a $1$-Calabi--Yau pair.
\end{defn}

Now we assume that $(X,\gamma)$ is a $d$-Calabi--Yau pair.

Let $D_J$ be as in \eqref{eq-def-wJ}.
For any $j\in J\subseteq\big\{1,\cdots,l\big\}$,
let $L_{J,j}$ be the normal line bundle of $D_J \hookrightarrow D_{J\backslash\{j\}}$.
For $J \subseteq\big\{1,\cdots,l\big\}$,
set
\begin{equation}
\label{eq-def-KJ}
K_J = K_X^d\big|_{D_J} \otimes \bigotimes_{j\in J} L_{J,j}^{-m_j}
= K_{D_J}^d \otimes \bigotimes_{j\in J} L_{J,j}^{-m_j-d} \;.
\end{equation}
which is a holomorphic line bundle over $D_J$.
In particular,
we have $K_\emptyset = K_X^d$.

Recall that $\mathrm{Res}_\cdot(\cdot)$ was defined in Definition \ref{def-res}.
By \eqref{eq-res-commute},
there exist
\begin{equation}
\Big(\gamma_J \in \mathscr{M}(D_J,K_J)\Big)_{J \subseteq\{1,\cdots,l\}}
\end{equation}
such that
\begin{equation}
\label{eq-def-gammaJ}
\gamma_\emptyset = \gamma \;,\hspace{5mm}
\gamma_J = \mathrm{Res}_{D_J}(\gamma_{J\backslash\{j\}})
\hspace{2.5mm} \text{ for } j\in J\subseteq\big\{1,\cdots,l\big\} \;.
\end{equation}

Let $\omega$ be a K{\"a}hler form on $X$.
Let $\big|\cdot\big|_{K_{D_J},\omega}$ be the metric on $K_{D_J}$ induced by $\omega$.
Let $\big|\cdot\big|_{L_{J,j},\omega}$ be the metric on $L_{J,j}$ induced by $\omega$.
Let $\big|\cdot\big|_{K_J,\omega}$ be the metric on $K_J$ induced by $\big|\cdot\big|_{K_{D_J},\omega}$ and $\big|\cdot\big|_{L_{J,j},\omega}$ via \eqref{eq-def-KJ}.

Let $g^{TD_J}_\omega$ be the metric on $TD_J$ induced by $\omega$.
Let $c_k\big(TD_J,g^{TD_J}_\omega\big)$
be $k$-th Chern form of $\big(TD_J,g^{TD_J}_\omega\big)$.
Let $n$ be the dimension of $X$.
Let $|J|$ be the number of elements in $J$.
Set
\begin{equation}
\label{eq-def-aJ}
a_J(\gamma,\omega) =
\frac{1}{12} \int_{D_J} c_{n-|J|}\big(TD_J,g^{TD_J}_\omega\big)
\log \big|\gamma_J\big|^{2/d}_{K_J,\omega} \;.
\end{equation}

We consider the short exact sequence of holomorphic vector bundles over $D_J$,
\begin{equation}
0 \rightarrow TD_J \rightarrow TD_{J\backslash\{j\}}\big|_{D_J} \rightarrow L_{J,j} \rightarrow 0 \;.
\end{equation}
Let $g^{TD_{J\backslash\{j\}}}_\omega$ be the metric on $TD_{J\backslash\{j\}}$ induced by $\omega$.
Let
\begin{equation}
\widetilde{c}\Big(TD_J,TD_{J\backslash\{j\}}\big|_{D_J},g^{TD_{J\backslash\{j\}}}_\omega\big|_{D_J}\Big)
\end{equation}
be the same Bott--Chern form as in \cite[\textsection 1.1]{z}.
Set
\begin{equation}
\label{eq-def-bIk}
b_{J,j}(\omega) =
\frac{1}{12} \int_{D_J}
\widetilde{c}\Big(TD_J,TD_{J\backslash\{j\}}\big|_{D_J},g^{TD_{J\backslash\{j\}}}_\omega\big|_{D_J}\Big) \;.
\end{equation}

Let $w_d^J$ be as in \eqref{eq-def-wJ}.
For ease of notations,
we denote $\tau_\mathrm{BCOV}(D_J,\omega) = \tau_\mathrm{BCOV}\big(D_J,\omega\big|_{D_J}\big)$.

The second author \cite[Definition 3.2]{z2} defined the following extended BCOV invariant.

\begin{defn}
\label{def-BCOV-Pair}
The BCOV invariant of a $d$-Calabi--Yau pair $(X,\gamma)$ is defined by
\begin{equation}
\label{eq-def-tau-gamma}
\tau_d(X,\gamma) = \sum_{J \subseteq\{1,\cdots,l\}}
w_d^J \bigg( \tau_\mathrm{BCOV}(D_J,\omega) - a_J(\gamma,\omega) - \sum_{j\in J} \frac{m_j+d}{d} b_{J,j}(\omega) \bigg) \;.
\end{equation}
It is shown in \cite[Theorem 3.1]{z2} that
$\tau_{d}(X, \gamma)$ is independent of $\omega$.
\end{defn}

\subsection{Projective spaces of dimension $1$ and $2$}
\label{subsec:P1andP2}

We identify $\CP^n$ with $\C^n \cup \CP^{n-1}$.
Let $(z_1,\cdots,z_n)\in\C^n$ be the affine coordinates.
For positive integers $m_1,\cdots,m_n$,
let $\gamma_{m_1,\cdots,m_n}\in\mathscr{M}\big(\CP^n,K_{\CP^n}^d\big)$ be such that
\begin{equation}
\label{eqn:gamma}
\gamma_{m_1,\cdots,m_n}\big|_{\C^n} = z_1^{m_1}\cdots z_n^{m_n} \big(\d z_1\wedge\cdots\wedge\d z_n\big)^d \;.
\end{equation}
Then $(\CP^n,\gamma_{m_1,\cdots,m_n})$ is a $d$-Calabi--Yau pair.

We denote
\begin{equation}
\label{eq-def-tau-P1}
\tau(\CP^n) = \tau_d\big(\CP^n,\gamma_{0,\cdots,0}\big) \;.
\end{equation}
By \cite[Proposition 3.3]{z2},
$\tau(\CP^n)$ is well-defined, i.e., independent of $d$.

\begin{thm}
\label{thm-tau-dim1}
For any $m\in\N$,
we have
\begin{equation}
\label{eq-thm-tau-dim1}
\tau_d\big(\CP^1,\gamma_m\big) = \tau(\CP^1) \;.
\end{equation}
In other words,
$\tau_d\big(\CP^1,\gamma_m\big)$ is independent of $m$.
\end{thm}
\begin{proof}
Let $w=1/z$.
We have
\begin{equation}
\label{eq1-pf-thm-tau-dim1}
\gamma_m=z^m (\d z)^d=\frac{(-1)^d}{w^{m+2d}}(\d w)^d \;.
\end{equation}
We have $\div(\gamma_m) = m \{0\} - (m+2d) \{\infty\}$.
Recall that $\mathrm{Res}_\cdot(\cdot)$ was defined in Definition \ref{def-res}.
We have
\begin{equation}
\label{eq2-pf-thm-tau-dim1}
\mathrm{Res}_{\{0\}}(\gamma_m) = (\d z)^{m+d} \;,\hspace{5mm}
\mathrm{Res}_{\{\infty\}}(\gamma_m) = (-1)^d (\d w)^{-m-d} \;.
\end{equation}

Let $\omega$ be a K{\"a}hler form on $\CP^1$.
We will use the notations in \eqref{eq-def-tau-gamma}.
Since $D_J = \mathrm{pt}$ for $|J|=1$ and $D_J = \emptyset$ for $|J|>1$,
we have
\begin{equation}
\label{eq3a-pf-thm-tau-dim1}
b_{J,j}(\omega) = 0 \hspace{4mm} \text{for any } J \text{ and } j \in J \;,\hspace{5mm}
\tau_\mathrm{BCOV}(D_J,\omega) = 0 \hspace{4mm} \text{for } |J| \geqslant 1 \;.
\end{equation}
Let $g^{T\CP^1}$ (resp. $g^{T^*\CP^1}$) be the metric on $T\CP^1$ (resp. $T^*\CP^1$) induced by $\omega$.
Let $\big|\d z\big|$ (resp. $\big|\d w\big|$) be the norm of $\d z$ (resp. $\d w$) with respect to $g^{T^*\CP^1}$.
By \eqref{eq-def-tau-gamma} and \eqref{eq1-pf-thm-tau-dim1}-\eqref{eq3a-pf-thm-tau-dim1},
We have
\begin{align}
\label{eq3-pf-thm-tau-dim1}
\begin{split}
& \tau_d\big(\CP^1,\gamma_m\big) \\
& = \tau_\mathrm{BCOV}\big(\CP^1,\omega\big)
- \frac{1}{12} \int_{\CP^1}c_1\big(T\CP^1,g^{T\CP^1}\big) \Big( \log\big|\d z\big|^2 + \frac{m}{d} \log\big|z\big|^2 \Big) \\
& \hspace{5mm} + \frac{m}{12} \log\big|\d z\big|^2_0 - \frac{m+2d}{12} \log\big|\d w\big|^2_\infty \;,
\end{split}
\end{align}
where the second term corresponds to $a_\emptyset(\gamma,\omega)$
and the last two terms correspond to $a_J(\gamma,\omega)$ with $|J|=1$.

In the sequel,
we take the Fubini--Study metric on $\CP^1$, whose K{\"a}hler form is
\begin{equation}
\label{eq4-pf-thm-tau-dim1}
\omega =  \frac{i \d z\wedge \d\overline{z}}{\big(1+|z|^2\big)^2} \;.
\end{equation}
We have
\begin{equation}
\label{eq5-pf-thm-tau-dim1}
c_1\big(T\CP^1,g^{T\CP^1}\big) = \frac{\omega}{\pi} \;,\hspace{5mm}
\big|\d z\big|^2 = \big(1+|z|^2\big)^2 \;.
\end{equation}
By \eqref{eq4-pf-thm-tau-dim1} and \eqref{eq5-pf-thm-tau-dim1},
we have
\begin{equation}
\label{eq6-pf-thm-tau-dim1}
\log\big|\d z\big|^2_0 = \log\big|\d w\big|^2_\infty = 0 \;,\hspace{5mm}
\int_{\CP^1}c_1\big(T\CP^1,g^{T\CP^1}\big)\log\big|z\big|^2 = 0 \;.
\end{equation}
By \eqref{eq3-pf-thm-tau-dim1} and \eqref{eq6-pf-thm-tau-dim1},
we obtain \eqref{eq-thm-tau-dim1}.
This completes the proof.
\end{proof}

\begin{thm}
\label{thm-tau-dim2}
For any $m_1,m_2\in\N$,
we have
\begin{equation}
\label{eq-thm-tau-dim2}
\tau_{d}\big(\CP^2,\gamma_{m_1,m_2}\big)
= \tau(\CP^2) + \Big( \frac{3}{2} - \frac{m_1}{m_1+d} - \frac{m_2}{m_2+d} - \frac{m_1+m_2+3d}{m_1+m_2+2d} \Big) \tau(\CP^1) \;.
\end{equation}
\end{thm}
\begin{proof}
Let $[\xi_0:\xi_1:\xi_2]$ be homogenous coordinates on $\CP^2$.
Let $H_1 \subseteq \CP^2$ (resp. $H_2 \subseteq \CP^2$, $H_\infty \subseteq \CP^2$)
be defined by $\xi_1=0$ (resp. $\xi_2=0$, $\xi_0=0$).
Set
\begin{equation}
\label{eq01-pf-thm-tau-dim2}
z_1 = \xi_1/\xi_0 \;,\hspace{2.5mm}
z_2 = \xi_2/\xi_0 \;,\hspace{2.5mm}
w_0=\xi_0/\xi_2 \;,\hspace{2.5mm}
w_1=\xi_1/\xi_2 \;,\hspace{2.5mm}
t_0=\xi_0/\xi_1 \;,\hspace{2.5mm}
t_2=\xi_2/\xi_1 \;.
\end{equation}
Then $(z_1,z_2)$ (resp. $(w_0,w_1)$, $(t_0,t_2)$) are affine coordinates
on $\CP^2\backslash H_\infty$ (resp. $\CP^2\backslash H_2$, $\CP^2\backslash H_1$).
We have
\begin{align}
\label{eq02-pf-thm-tau-dim2}
\begin{split}
\gamma_{m_1,m_2} & = z_1^{m_1}z_2^{m_2}\big(\d z_1\wedge \d z_2\big)^d \\
& = \frac{w_1^{m_1}}{w_0^{m_1+m_2+3d}} \big(\d w_0\wedge \d w_1\big)^d
= \frac{t_2^{m_2}}{t_0^{m_1+m_2+3d}} \big(\d t_2\wedge \d t_0\big)^d \;.
\end{split}
\end{align}
We remark that $\div(\gamma_{m_1,m_2}) = m_1 H_1 + m_2 H_2 - (m_1+m_2+3d) H_\infty$.

Recall that $\mathrm{Res}_\cdot(\cdot)$ was defined in Definition \ref{def-res}.
We have
\begin{align}
\label{eq04-pf-thm-tau-dim2}
\begin{split}
\mathrm{Res}_{H_1}(\gamma_{m_1,m_2}) & = z_2^{m_2}\big(\d z_1\big)^{m_1+d}\big(\d z_2\big)^d \;,\\
\mathrm{Res}_{H_2}(\gamma_{m_1,m_2}) & = z_1^{m_1}\big(\d z_2\big)^{m_2+d}\big(\d z_1\big)^d \;,\\
\mathrm{Res}_{H_\infty}(\gamma_{m_1,m_2}) & = w_1^{m_1}\big(\d w_0\big)^{-m_1-m_2-2d}\big(\d w_1\big)^d \;,\\
\mathrm{Res}_{H_1 \cap H_2}\big(\mathrm{Res}_{H_1} (\gamma_{m_1,m_2})\big) & = \big(\d z_1\big)^{m_1+d}\big(\d z_2\big)^{m_2+d} \;,\\
\mathrm{Res}_{H_1 \cap H_\infty}\big(\mathrm{Res}_{H_\infty}(\gamma_{m_1,m_2})\big) & = \big(\d w_0\big)^{-m_1-m_2-2d}\big(\d w_1\big)^{m_1+d} \;,\\
\mathrm{Res}_{H_2 \cap H_\infty}\big(\mathrm{Res}_{H_\infty}(\gamma_{m_1,m_2})\big) & = \big(\d t_0\big)^{-m_1-m_2-2d}\big(\d t_2\big)^{m_2+d} \;.
\end{split}
\end{align}

We fix a Fubini--Study metric on $\CP^2$, whose K{\"a}hler form is as follows:
\begin{equation}
\label{eq11-pf-thm-tau-dim2}
\omega =
\frac{i \big(\d z_1\wedge \d\overline{z}_1 + \d z_2\wedge \d\overline{z}_2
-\overline{z}_1z_2\d z_1\wedge\d\overline{z}_2
-z_1\overline{z}_2\d\overline{z}_1\wedge\d z_2\big) }
{\big(1+|z_1|^2+|z_2|^2\big)^2} \;.
\end{equation}
We will use the notations in \eqref{eq-def-tau-gamma}.
With the K{\"a}hler form \eqref{eq11-pf-thm-tau-dim2},
we have
\begin{equation}
\label{eq11a-pf-thm-tau-dim2}
a_J(\gamma_{m_1,m_2},\omega) = b_{J,j}(\omega) = 0
\hspace{5mm} \text{for } |J| = 2 \;,\hspace{5mm}
b_{J,j}(\omega) = 0 \hspace{5mm} \text{for } |J| = 1 \;.
\end{equation}
By \eqref{eq-def-tau-gamma}, \eqref{eq11-pf-thm-tau-dim2}, \eqref{eq11a-pf-thm-tau-dim2}
and the fact that $\tau_\mathrm{BCOV}(\mathrm{pt}) = 0$,
we have

\begin{align}
\label{eq12-pf-thm-tau-dim2}
\begin{split}
& \tau_d\big(\CP^2,\gamma_{m_1,m_2}\big) \\
& = \tau_\mathrm{BCOV}\big(\CP^2,\omega\big)
- \frac{1}{12} \frac{1}{d} \int_{\CP^2}c_2\big(T\CP^2,g^{T\CP^2}\big)\log\Big|z_1^{m_1}z_2^{m_2}\big(\d z_1\wedge\d z_2\big)^d\Big|^2 \\
& \hspace{5mm} - \frac{m_1}{m_1+d} \bigg( \tau_\mathrm{BCOV}\big(H_1,\omega\big) \\
& \hspace{30mm} - \frac{1}{12} \frac{1}{d} \int_{H_1}c_1\big(TH_1,g^{TH_1}\big)\log\Big|z_2^{m_2}\big(\d z_1\big)^{m_1+d}\big(\d z_2\big)^d\Big|^2 \bigg) \\
& \hspace{5mm} - \frac{m_2}{m_2+d} \bigg( \tau_\mathrm{BCOV}\big(H_2,\omega\big) \\
& \hspace{30mm} - \frac{1}{12} \frac{1}{d} \int_{H_2}c_1\big(TH_2,g^{TH_2}\big)\log\Big|z_1^{m_1}\big(\d z_2\big)^{m_2+d}\big(\d z_1\big)^d\Big|^2 \bigg) \\
& \hspace{5mm} - \frac{m_1+m_2+3d}{m_1+m_2+2d} \bigg( \tau_\mathrm{BCOV}\big(H_\infty,\omega\big) \\
& \hspace{20mm} - \frac{1}{12} \frac{1}{d} \int_{H_\infty}c_1\big(TH_\infty,g^{TH_\infty}\big)\log\Big|w_1^{m_1}\big(\d w_0\big)^{-m_1-m_2-2d}\big(\d w_1\big)^d\Big|^2 \bigg) \;.
\end{split}
\end{align}
By \eqref{eq-def-tau-P1}, \eqref{eq3-pf-thm-tau-dim1}, \eqref{eq6-pf-thm-tau-dim1} and \eqref{eq12-pf-thm-tau-dim2},
we have
\begin{align}
\label{eq13-pf-thm-tau-dim2}
\begin{split}
& \tau_d\big(\CP^2,\gamma_{m_1,m_2}\big) \\
& = \tau_\mathrm{BCOV}\big(\CP^2,\omega\big)
- \frac{1}{12} \int_{\CP^2}c_2\big(T\CP^2,g^{T\CP^2}\big)\log\big|\d z_1\wedge\d z_2\big|^2 \\
& \hspace{5mm} - \frac{1}{12} \frac{1}{d} \int_{\CP^2}c_2\big(T\CP^2,g^{T\CP^2}\big) \Big( m_1 \log\big|z_1\big|^2 + m_2 \log\big|z_2\big|^2 \Big) \\
& \hspace{5mm} - \frac{m_1}{m_1+d} \bigg( \tau(\CP^1) - \frac{1}{12} \frac{m_1+d}{d} \int_{H_1}c_1\big(TH_1,g^{TH_1}\big)\log\big|\d z_1\big|^2 \bigg) \\
& \hspace{5mm} - \frac{m_2}{m_2+d} \bigg( \tau(\CP^1) - \frac{1}{12} \frac{m_2+d}{d} \int_{H_2}c_1\big(TH_2,g^{TH_2}\big)\log\big|\d z_2\big|^2 \bigg) \\
& \hspace{5mm} - \frac{m_1+m_2+3d}{m_1+m_2+2d} \bigg( \tau(\CP^1) + \frac{1}{12} \frac{m_1+m_2+2d}{d} \int_{H_\infty}c_1\big(TH_\infty,g^{TH_\infty}\big)\log\big|\d w_0\big|^2 \bigg) \;.
\end{split}
\end{align}

Similarly to \eqref{eq6-pf-thm-tau-dim1},
we have
\begin{equation}
\label{eq21-pf-thm-tau-dim2}
\int_{\CP^2}c_2\big(T\CP^2,g^{T\CP^2}\big) \log\big|z_1\big|^2 =
\int_{\CP^2}c_2\big(T\CP^2,g^{T\CP^2}\big) \log\big|z_2\big|^2 = 0 \;.
\end{equation}
On the other hand,
by \eqref{eq11-pf-thm-tau-dim2},
we have
\begin{align}
\label{eq22-pf-thm-tau-dim2}
\begin{split}
\int_{H_1}c_1\big(TH_1,g^{TH_1}\big)\log\big|\d z_1\big|^2
& = \int_{H_2}c_1\big(TH_2,g^{TH_2}\big)\log\big|\d z_2\big|^2 \\
& = \int_{H_\infty}c_1\big(TH_\infty,g^{TH_\infty}\big)\log\big|\d w_0\big|^2 \;.
\end{split}
\end{align}
From \eqref{eq13-pf-thm-tau-dim2}-\eqref{eq22-pf-thm-tau-dim2},
we obtain \eqref{eq-thm-tau-dim2}.
This completes the proof.
\end{proof}

\subsection{Projective bundle}
\label{subsec:ProjBundle}

Let $Y$ be a compact K{\"a}hler manifold.
Let $N$ be a holomorphic vector bundle of rank $r$ over $Y$.
Set
\begin{equation}
\label{eqn:X=CompletedProjBun}
X = \mathbb{P}(N\oplus\mathds{1}) \;.
\end{equation}
Let $\mathcal{N}$ be the total space of $N$.
We have $X = \mathcal{N} \cup \mathbb{P}(N)$.

Let $s\in\{0,\cdots,r\}$.
Let $\big(L_j\big)_{j=1,\cdots,s}$ be holomorphic line bundles over $Y$.
We assume that
there is a surjective map
\begin{equation}
\label{eq-Nsurj}
N \rightarrow L_1 \oplus \cdots \oplus L_s \;.
\end{equation}
Let $N^*$ be the dual of $N$.
Taking the dual of \eqref{eq-Nsurj},
we get
\begin{equation}
\label{eq-Ninj}
L_1^{-1} \oplus \cdots \oplus L_s^{-1} \hookrightarrow N^* \;.
\end{equation}
Let $m_1,\cdots,m_s$ be positive integers.
Let $d\in\N\backslash\{0\}$.
Let
\begin{equation}
\label{eq-def-gammaY}
\gamma_Y \in \mathscr{M}\big(Y,K_Y^d \otimes (\det N^*)^d \otimes L_1^{-m_1} \otimes \cdots \otimes L_s^{-m_s}\big)
\end{equation}
be an invertible element.
We assume that
\begin{itemize}
\item[-] $\div(\gamma_Y)$ is of simple normal crossing support;
\item[-] $\div(\gamma_Y)$ does not possess components of multiplicity $-d$.
\end{itemize}
We denote $m = m_1+\cdots+m_s$.
Let $S^mN^*$ be the $m$-th symmetric tensor power of $N^*$.
By \eqref{eq-Ninj} and \eqref{eq-def-gammaY},
we have
\begin{equation}
\label{eq-gammaY}
\gamma_Y \in \mathscr{M}\big(Y,K_Y^d \otimes (\det N^*)^d \otimes S^m N^*\big) \;.
\end{equation}
Let $\pi: X = \mathbb{P}(N\oplus\mathds{1}) \rightarrow Y$ be the canonical projection.
We have
\begin{equation}
\label{eq-KX-pi}
K_X\big|_\mathcal{N} = \pi^* \big( K_Y\otimes \det N^* \big) \;.
\end{equation}
We may view a section of $S^mN^*$ as a function on $\mathcal{N}$.
By \eqref{eq-gammaY} and \eqref{eq-KX-pi},
$\gamma_Y$ may be viewed as an element of $\mathscr{M}(\mathcal{N},K_X^d)$.
Let $\gamma_X\in\mathscr{M}(X,K_X^d)$ be such that $\gamma_X\big|_\mathcal{N} = \gamma_Y$.

For $j=1,\cdots,s$,
set
\begin{equation}
N_j = \mathrm{Ker}\big(N\rightarrow L_j\big) \;,\hspace{5mm}
X_j = \mathbb{P}(N_j\oplus\mathds{1}) \subseteq X \;,\hspace{5mm}
X_\infty = \mathbb{P}(N) \subseteq X \;.
\end{equation}
We have
\begin{equation}
\div(\gamma_X) =
\pi^* \div(\gamma_Y) - (m+rd+d) X_\infty + \sum_{j=1}^s m_j X_j \;.
\end{equation}
Hence $(X,\gamma_X)$ is a $d$-Calabi--Yau pair.

Let $Z$ be the fiber of $\pi: X \rightarrow Y$.
Let $U\subseteq Y$ be a small open subset.
We fix an identification $\pi^{-1}(U) = U \times Z$ such that
there exist $\gamma_U\in \mathscr{M}(U,K_U^d)$ and $\gamma_Z\in\mathscr{M}(Z,K_Z^d)$ satisfying
\begin{equation}
\label{eqn:GammaX}
\gamma_X\big|_{\pi^{-1}(U)} = \mathrm{pr}_1^*(\gamma_U) \otimes \mathrm{pr}_2^*(\gamma_Z) \;.
\end{equation}
Then $(Z,\gamma_Z)$ is a $d$-Calabi--Yau pair.

The following theorem was proved by the second author \cite[Theorem 3.6]{z2}

\begin{thm}
\label{thm-proj-bd}
The following identity holds,
\begin{equation}
\tau_d\big(X,\gamma_X\big)
= \chi_d\big(Y,\gamma_Y\big) \tau_{d}\big(Z,\gamma_Z\big) \;.
\end{equation}
\end{thm}

\subsection{Blow-up}

Let $(X,\gamma)$ be a $d$-Calabi--Yau pair.
We denote
\begin{equation}
\div(\gamma) = D = \sum_{j=1}^l m_jD_j \;.
\end{equation}

Let $Y \subseteq X$ be a connected complex submanifold
intersecting $D_1,\cdots,D_l$ transversally (see Definition \ref{def-ti}).
Assume that
for $j\in\{1,\cdots,l\}$ satisfying $Y \subseteq D_j$,
we have $m_j > 0$.
Let $r$ be the codimension of $Y\subseteq X$.
Let $s$ be the number of $D_j$ containing $Y$.
We have $s\leqslant r$.
Without loss of generality,
we assume that
\begin{equation}
Y \subseteq D_j \hspace{2.5mm} \text{for } j=1,\cdots,s \;;\hspace{5mm}
Y \nsubseteq D_j \hspace{2.5mm} \text{for } j=s+1,\cdots,l \;.
\end{equation}

Let $f: X' \rightarrow X$ be the blow-up along $Y$.
Let $D_j' \subseteq X'$ be the strict transformation of $D_j \subseteq X$.
Set $E = f^{-1}(Y)$.
We denote $D'=\div(f^*\gamma)$.
We denote
\begin{equation}
m_0 = m_1 + \cdots + m_s + rd -d \;.
\end{equation}
We have
\begin{equation}
D' = m_0 E + \sum_{j=1}^l m_j D_j' \;.
\end{equation}
Hence $(X',f^*\gamma)$ is a $d$-Calabi-Yau pair.

Set
\begin{equation}
D_Y = \sum_{j=s+1}^l m_j (D_j \cap Y) \;,\hspace{5mm}
D_E = \sum_{j=s+1}^l m_j (D_j' \cap E) \;,
\end{equation}
which are divisors with simple normal crossing support.

Now we introduce a pair involving projective spaces, which is a slightly different way of denoting the one defined in Equation \eqref{eqn:gamma}. More precisely: we identify $\CP^r$ with $\C^r \cup \CP^{r-1}$.
Let $(z_1,\cdots,z_r)\in\C^r$ be the coordinates.
Let $\gamma_{r,m_1,\cdots,m_s}\in\mathscr{M}(\CP^r,K_{\CP^r}^d)$ be such that
\begin{equation}
\label{eq-gammaZ}
\gamma_{r,m_1,\cdots,m_s} \big|_{\C^r} = \big(\d z_1\wedge\cdots\wedge\d z_r \big)^d \prod_{j=1}^s z_j^{m_j} \;.
\end{equation}
Then $(\CP^r,\gamma_{r,m_1,\cdots,m_s})$ is a $d$-Calabi-Yau pair. Note that $\gamma_{r,m_1,\cdots,m_s}$ is precisely $\gamma_{m_1, \cdots, m_s, 0,0, \cdots, 0}$ with $r-s$ zeros in the end, in the notation of \eqref{eqn:gamma}.

The following blow-up formula was proved by the second author \cite[Theorem 0.5]{z2}.

\begin{thm}
\label{thm-bl}
The following identity holds,
\begin{align}
\label{eq-thm-bl}
\begin{split}
& \tau_d(X',f^*\gamma) - \tau_d(X,\gamma) \\
& = \chi_d(E,D_E) \tau_d\big(\CP^1,\gamma_{1,m_0}\big)
- \chi_d(Y,D_Y) \tau_d\big(\CP^r,\gamma_{r,m_1,\cdots,m_s}\big) \;.
\end{split}
\end{align}
\end{thm}

\begin{rem}
\label{rem-bl}
Keep the notation as is  in Theorem \ref{thm-bl}.
Let $g: \Bl_0\CP^r \rightarrow \CP^r$ be the blow-up along $0\in\C^r\subseteq\CP^r$.
Since $\Bl_0\CP^r$ is a $\CP^1$-bundle over $\CP^{r-1}$, apply Theorem \ref{thm-proj-bd}, with $N=L=\mathcal{O}_{\CP^{r-1}}(-1)$,
\begin{equation}
\tau_d\big(\Bl_0\CP^r,g^*\gamma_{m_1,\cdots,m_s}\big) =
\chi_d\big(\CP^{r-1},D_{m_1,\cdots,m_s}\big) \tau_d\big(\CP^1,\gamma_{m_0}\big) \;.
\end{equation}
Note that $\chi_d(E,D_E) = \chi_d(Y,D_Y)\chi_d\big(\CP^{r-1},D_{m_1,\cdots,m_s}\big)$,
we could reinterpret Theorem \ref{thm-bl} as follows,
\begin{align}
\begin{split}
& \tau_d\big(X',f^*\gamma_X\big) - \tau_d\big(X,\gamma_X\big) \\
& = \chi_d\big(Y,D_Y\big)
\Big( \tau_d\big(\Bl_0\CP^r,g^*\gamma_{m_1,\cdots,m_s}\big)
- \tau_d\big(\CP^r,\gamma_{m_1,\cdots,m_s}\big) \Big) \;.
\end{split}
\end{align}
\end{rem}

\section{Motivic integration and BCOV invariants}
\label{ch-motivic}
In this section,
we use the theory of motivic integration to explain our key construction, namely, the BCOV invariant for pairs $\tau(X, \gamma)$.
We stress the fact that this section is purely heuristic and logically independent of the rest of the paper.

\subsection{Motivic integration}
\label{subsect:MotInt}
We consider an $n$-dimensional smooth complex algebraic variety $X$
and an effective divisor with simple normal crossing support $D=\sum_{j=1}^l m_jD_j$ on $X$.
Denote by $\mathcal{L}_{\infty}(X)$ the space of formal arcs in $X$,
that is,
the projective limit of jet schemes $\mathcal{L}_n(X):=X\left(\C[t]/(t^{n+1})\right)$ (see \cite[\S 1]{DenefLoeser99Inv}).
Let
\begin{equation}
\operatorname{ord}_{D}: \mathcal{L}_{\infty}(X)\to \N\cup\{+\infty\}
\end{equation}
be the function sending a formal arc to its intersection number with the divisor $D$.

Let $\Var_\C$ be the category of complex algebraic varieties.
The Grothendieck group of complex algebraic varieties, denoted by $K_0(\Var_\C)$,
is the free abelian group generated by the isomorphism classes of objects in $\Var_\C$,
modulo the \emph{scissor relation}:
\begin{equation}
\label{eq-scissor}
[X] = [Y] + [X\backslash Y] \hspace{2.5mm} \text{for any } X \text{ and any closed subvariety } Y \subseteq X \;.
\end{equation}
$K_{0}(\Var_{\C})$ is endowed with a natural ring structure given by the fiber product.

Let $\mathbb{L}$ be the class of the affine line.
We denote $\mathcal{M} = K_{0}(\Var_{\C})[\mathbb{L}^{-1}]$,
the localization of $K_0(\Var_\C)$ with respect to the multiplicative system $\big\{\mathbb{L}^k\big\}_{k\in \N}$.
For any integer $i$,
let $F^i\mathcal{M} \subseteq \mathcal{M}$ be the subgroup generated by elements of the form $\mathbb{L}^{-m}[Y]$ with $m-\dim(Y)\geqslant i$.
Then $F^\bullet$ is a filtration on $\mathcal{M}$.
Let $\widehat{\mathcal{M}}$ be the completion of $\mathcal{M}$ with respect to $F^{\bullet}$.
The \emph{motivic Igusa zeta function}\footnote{It is usually denoted by $Z(X, \mathcal{I}_{D}; T)$,
where $\mathcal{I}_{D}=\mathcal{O}_{X}(-D)$ is the ideal sheaf of $D$.} is by definition
\begin{equation}
Z(X, D; T):=\int_{\mathcal{L}_{\infty}(X)}T^{\operatorname{ord_{D}}}\d\mu \in \widehat{\mathcal{M}}[[T]] \;,
\end{equation}
where $\mu$ is the motivic measure constructed by Kontsevich \cite{Kontsevich}
and Denef--Loeser \cite[Definition-Proposition 3.2]{DenefLoeser99Inv}.

The following theorem gives a formula for $Z(X, D; T)$,
see \cite[Theorem 3.3.4]{MotIntBook}.

\begin{thm}
The following identity holds,
\begin{equation}
\label{eq-MotIgusa}
Z(X,D;T) = \sum_{J\subseteq\{1,\dots,l\}} \mathbb{L}^{|J|-n}\left(\prod_{j\in J} \frac{1-T^{-m_j}}{\mathbb{L} T^{-m_j}-1}\right) [D_J] \;,
\end{equation}
where $D_{J}=\bigcap_{j\in J}D_j$ with the convention that $D_{\emptyset}=X$.
\end{thm}

Let $d$ be a positive integer.
We define
\begin{equation}
F_{d}(X, D) := Z(X, D; \mathbb{L}^{-1/d})
= \sum_{m=0}^{\infty} \mu\big(\operatorname{ord}_{D}^{-1}(m)\big) \mathbb{L}^{-m/d}
\;\in \widehat{\mathcal{M}}[\mathbb{L}^{1/d}] \;.
\end{equation}
By \eqref{eq-MotIgusa}, we have
\begin{equation}
\label{eq-MotIntDef}
F_{d}(X, D) = \sum_{J\subseteq\{1,\dots,l\}} \mathbb{L}^{|J|-n}\left(\prod_{j\in J} \frac{1-\mathbb{L}^{m_j/d}}{\mathbb{L}^{1+m_j/d}-1}\right)[D_J]\;.
\end{equation}

An equivalent form of \eqref{eq-MotIntDef} when $d=1$ is in Craw \cite[Theorem 1.1]{Craw04}.

Now we state the formula of change of variables,
due to Kontsevich \cite{Kontsevich} and Denef--Loeser \cite[Theorem 1.16]{DenefLoeserCompositio02}),
in the following form taken from Craw \cite[Theorem 2.19]{Craw04} (when $d=1$).

\begin{thm}
\label{thm-MotInt}
Let $X$ be a projective complex manifold.
Let $f: X'\to X$ be the blow-up along a smooth center.
Let $K_{X'/X}$ be the relative canonical divisor.
Let $d$ be a positive integer.
Let $D$ be an effective divisor on $X$ such that
both $D$ and $f^*(D)+dK_{X'/X}$ are of simple normal crossing support.
We have
\begin{equation}
F_{d}(X, D)=F_{d}\big(X', f^*D +d K_{X'/X}\big) \;.
\end{equation}
\end{thm}

\subsection{From motivic integration to BCOV invariant}
\label{subsec-Heuristic}
Let $X$ be a smooth projective complex variety.
Let $\gamma$ be a $d$-canonical form on $X$ such that
$D = \div(\gamma)$ satisfies the condition $(\star_{d})$ in Definition \ref{def-LocInv}.
Hence $(X,\gamma)$ is a $d$-Calabi--Yau pair in the sense of Definition \ref{def-cy-pair}.

Recall that the Hodge realization is the ring homomorphism
\begin{equation}
\chi_{\Hdg} \colon K_0(\Var_\C) \to K_0(\operatorname{HS})
\end{equation}
that sends the class of a smooth projective variety $X$ to the class of its cohomology $H^{\bullet}(X, \Z)$ endowed with Hodge structure.
It is easy to see that $\chi_{\Hdg}(\mathbb{L})=\Z(-1)$ is the Lefschetz Hodge structure, which we will denote by $\mathbf{L}$ in the sequel. Therefore,
for a Hodge structure $H^\bullet$ and $s\in\Z$,
$\mathbf{L}^{s}H^{\bullet}$ is the Tate twist $H^{\bullet}(-s)$, namely,
\begin{equation}
\mathbf{L}^sH^k_\Z = H^{k-2s}_\Z \;,\hspace{5mm}
\mathbf{L}^sH^{p,q}_\C = H^{p-s,q-s}_\C \;.
\end{equation}
For a polynomial $f(x)=a_0 + a_1x + \cdots + a_mx^m\in \Z[x]$,
we denote
\begin{equation}
f(\mathbf{L})H^\bullet = \sum_{s=0}^m \big(\mathbf{L}^sH^\bullet\big)^{\oplus a_s} \;.
\end{equation}
By \eqref{eq-MotIntDef},
we have
\begin{equation}
\label{eq-MotInt-Hdg}
\chi_{\Hdg}\left(F_{d}(X,D)\mathbb{L}^{\frac{n}{2}}\right) =
\sum_J \mathbf{L}^{|J|-\frac{n}{2}} \bigg( \prod_{j\in J} \frac{1-\mathbf{L}^{m_j/d}}{\mathbf{L}^{1+m_j/d}-1} \bigg) H^\bullet(D_J) \;.
\end{equation}
Mimicking \eqref{eq-def-eta-dR} and \eqref{eq-def-lambda-dr}, for a Hodge structure $H^\bullet$, we define
\begin{equation}
\eta(H^\bullet)=\bigotimes_k\left(\det H^k\right)^{(-1)^k} \;,\hspace{5mm}
\quad \lambda_\mathrm{dR}(H^\bullet)= \bigotimes_k \left(\det H^k\right)^{(-1)^kk} \;.
\end{equation}
We are interested in applying $\lambda_\mathrm{dR}$ to \eqref{eq-MotInt-Hdg}.
First we remark that
\begin{equation}
\eta\big(\mathbf{L} H^\bullet\big) = \eta(H^\bullet) \;,\hspace{5mm}
\lambda_\mathrm{dR}\big(\mathbf{L} H^\bullet\big) = \Big(\eta(H^\bullet)\Big)^{2} \otimes \lambda_\mathrm{dR}(H^\bullet)  \;.
\end{equation}
Therefore, for any polynomial $f$,
we have
\begin{equation}
\lambda_\mathrm{dR}\big(f(\mathbf{L})H^\bullet\big) =
\Big(\eta(H^\bullet)\Big)^{2f'(1)} \otimes \Big(\lambda_\mathrm{dR}(H^\bullet)\Big)^{f(1)} \;.
\end{equation}
Let $w_d^J$ be as in \eqref{eq-def-wJ}.
For
\begin{equation}
f(x) = x^{|J|-\frac{n}{2}} \prod_{j\in J} \frac{1-x^{m_j/d}}{x^{1+m_j/d}-1} \;,
\end{equation}
we have
\begin{equation}
\label{eq-lambdadR-f}
\lambda_\mathrm{dR}\big(f(\mathbf{L})H^\bullet\big) =
\Big(\eta(H^\bullet)\Big)^{(|J|-n)w^{J}_{d} } \otimes \Big(\lambda_\mathrm{dR}(H^\bullet)\Big)^{w^{J}_{d}} \;.
\end{equation}
By \eqref{eq-MotInt-Hdg} and \eqref{eq-lambdadR-f},
we have
\begin{align}
\label{eq-MotInt-lambda}
\begin{split}
& \lambda_\mathrm{dR}\bigg(\chi_{\Hdg}\left(F_{d}(X, D)\mathbb{L}^{\frac{n}{2}}\right)\bigg)\\
& = \bigotimes_J \bigg( \Big(\lambda_\mathrm{dR}\big(H^\bullet(D_J)\big)\Big)^{w_d^{J}}
\otimes \Big(\eta\big(H^\bullet(D_J)\big)\Big)^{(|J|-n)w_d^{J}} \bigg) \;.
\end{split}
\end{align}
Observe that
the BCOV invariant $\tau_{d}(X, D)$ (cf.~Definition \ref{def-BCOV-Pair}) is essentially the Quillen metric on
$\bigotimes_J  \Big(\lambda_\mathrm{dR}\big(H^\bullet(D_J)\big)\Big)^{w_d^{J}}$.
On the other hand,
by Definition \ref{def-tau-top},
the Quillen metric on $\eta\big(H^\bullet(D_J)\big)$ gives rise to $\tau_\mathrm{top}(D_J)$,
which vanishes by \eqref{eq-prop-rei}.
Therefore,
our BCOV invariant is essentially the Quillen metric on the determinant line \eqref{eq-MotInt-lambda}.

\section{Birational BCOV invariants}
\label{ch-birat}


\begin{defn}
\label{def-tau-b}
For any $d$-Calabi--Yau pair $(X, \gamma)$,  its \emph{birational BCOV invariant} is
\begin{align}
\label{eq-def-tau-b}
\begin{split}
\taub\big(X,\gamma\big) & = \tau_d\big(X,\gamma\big)
- \frac{1}{2} \tau(\CP^1) \chi'_{d}\big(X,\gamma\big) \\
& \hspace{5mm} + \Big(-\frac{1}{2}\tau(\CP^2) + \frac{3}{4} \tau(\CP^1) \Big) \chi''_d\big(X,\gamma\big) \;,
\end{split}
\end{align}
where $\chi_d'(X, \gamma)$ and $\chi_d''(X, \gamma)$ are as in Definition \ref{def-LocInv},
applied to the invariants $\chi'$ and $\chi''$ in Example \ref{ex-localizable}.
\end{defn}

For a Calabi--Yau manifold $X$ and a $d$-canonical form $\gamma$ on $X$
such that $\int_X \big|\gamma\overline{\gamma}\big|^{1/d} = (2\pi)^{\dim X}$,
we have
\begin{equation}
\label{eq4-proof-thm-birat}
\tau_d^{\mathrm{bir}}(X, \gamma) =
\tau(X) - \frac{1}{2} \tau(\CP^1) \chi'(X)
+ \Big( - \frac{1}{2}\tau(\CP^2) + \frac{3}{4} \tau(\CP^1) \Big) \chi''(X) \;.
\end{equation}

\begin{lemme}
\label{lem-proj-bd-b}
Let $(X,\gamma_X)$, $(Y,\gamma_Y)$ and $(Z,\gamma_Z)$
be as in Theorem \ref{thm-proj-bd} (see \eqref{eqn:X=CompletedProjBun}-\eqref{eqn:GammaX} in Section \ref{subsec:ProjBundle}).
Then
\begin{align}
\begin{split}
\taub\big(X,\gamma_X\big) = \chi_{d}\big(Y,\gamma_Y\big) \taub\big(Z,\gamma_Z\big) \;.
\end{split}
\end{align}
\end{lemme}
\begin{proof}
Since $\chi'$ and $\chi''$ are log-type localizable invariants (Examples \ref{ex-localizable}),
Proposition \ref{prop-ProjBun} yields $\chi_d'(X,\gamma_X) = \chi_d(Y,\gamma_Y) \chi_d'(Z,\gamma_Z)$
and $\chi_d''(X,\gamma_X) = \chi_d(Y,\gamma_Y) \chi_d''(Z, \gamma_Z)$.
Combining them with Theorem \ref{thm-proj-bd} allows us to conclude.
\end{proof}

\begin{lemme}
\label{lem-bl-b}
Let $(X,\gamma_X)$, $(Y,D_Y)$, $f:X'\rightarrow X$
and $m_0,\cdots,m_s\in\Z$,
be as in Theorem \ref{thm-bl}.
We have
\begin{align}
\begin{split}
& \taub\big(X',f^*\gamma_X\big) - \taub\big(X,\gamma_X\big) \\
& = \chi_d\big(Y,D_Y\big)
\Big( \chi_d\big(\CP^{r-1},D_{m_1,\cdots,m_s}\big) \taub\big(\CP^1,\gamma_{m_0}\big)
- \taub\big(\CP^r,\gamma_{m_1,\cdots,m_s}\big) \Big) \;.
\end{split}
\end{align}
\end{lemme}
\begin{proof}
Similarly to the proof of Lemma \ref{lem-proj-bd-b},
we use Theorem \ref{thm-bl}, Remark \ref{rem-bl} and Proposition \ref{prop-Blowup}.
\end{proof}

\begin{thm}
\label{thm-tau-b}
For any $m_1,\cdots,m_n\in\N$,
we have
\begin{equation}
\label{eq-thm-blb}
\taub\big(\CP^n,\gamma_{m_1,\cdots,m_n}\big) = 0 \;.
\end{equation}
\end{thm}
\begin{proof}
We have (see Example \ref{ex-localizable})
\begin{align}
\label{eq1-pf-thm-tau-b}
\begin{split}
\chi(\CP^1) = 2 \;,\hspace{5mm}
\chi'(\CP^1) = 2 \;,\hspace{5mm}
\chi''(\CP^1) = 0 \;,\\
\chi(\CP^2) = 3 \;,\hspace{5mm}
\chi'(\CP^2) = 6 \;,\hspace{5mm}
\chi''(\CP^2) = 2 \;.
\end{split}
\end{align}
By Definition \ref{def-LocInv} and \eqref{eq1-pf-thm-tau-b},
for any $m_{1}, m_{2}\in \N$,
we have
\begin{align}
\label{eq2-pf-thm-tau-b}
\begin{split}
& \chi_d''\big(\CP^1,\gamma_{m_1}\big) = 0 \;,\hspace{5mm}
\chi_d''\big(\CP^2,\gamma_{m_1,m_2}\big) = 2 \;,\hspace{5mm}
\chi_d'\big(\CP^1,\gamma_{m_1}\big) = 2 \;,\\
& \chi_d'\big(\CP^2,\gamma_{m_1,m_2}\big) = 6 -  2 \Big( \frac{m_1}{m_1+d} + \frac{m_2}{m_2+d} + \frac{m_1+m_2+3d}{m_1+m_2+2d} \Big) \;.
\end{split}
\end{align}
By Theorem \ref{thm-tau-dim1}, Theorem \ref{thm-tau-dim2}, Definition \ref{def-tau-b} and \eqref{eq2-pf-thm-tau-b},
we have  $\taub\big(\CP^1,\gamma_{m_1}\big) = \taub\big(\CP^2,\gamma_{m_1,m_2}\big) = 0$.
Hence \eqref{eq-thm-blb} holds for $n \leqslant 2$.
We proceed by induction. Let $r\geqslant 2$ be an integer.
Assume that
\begin{equation}
\label{eq4-pf-thm-tau-b}
\taub\big(\CP^n,\gamma_{m_1,\cdots,m_n}\big) = 0 \hspace{2.5mm} \text{for }n \leqslant r \;.
\end{equation}
Let $i:\CP^1\hookrightarrow\CP^{r+1}$ be the extension of
$\C \ni z \mapsto (z,0,\cdots,0) \in \C^{r+1}$.
Let $f: \Bl_{\CP^1}\CP^{r+1}\rightarrow\CP^{r+1}$
be the blow-up along $i(\CP^1)$.
Then $\Bl_{\CP^1}\CP^{r+1}$ is a $\CP^2$-bundle over $\CP^{r-1}$.
By Lemma \ref{lem-proj-bd-b} and \eqref{eq4-pf-thm-tau-b},
we have
\begin{equation}
\label{eq7-pf-thm-tau-b}
\taub\big(\Bl_{\CP^1}\CP^{r+1},f^*\gamma_{m_1,\cdots,m_{r+1}}\big) = 0 \;.
\end{equation}
By Lemma \ref{lem-bl-b} and \eqref{eq4-pf-thm-tau-b},
we have
\begin{equation}
\label{eq8-pf-thm-tau-b}
\taub\big(\Bl_{\CP^1}\CP^{r+1},f^*\gamma_{m_1,\cdots,m_{r+1}}\big) -
\taub\big(\CP^{r+1},\gamma_{m_1,\cdots,m_{r+1}}\big) = 0 \;.
\end{equation}
From \eqref{eq7-pf-thm-tau-b} and \eqref{eq8-pf-thm-tau-b},
we obtain
\begin{equation}
\taub\big(\CP^n,\gamma_{m_1,\cdots,m_n}\big) = 0 \hspace{2.5mm} \text{for }n \leqslant r+1 \;.
\end{equation}
This completes the proof by induction.
\end{proof}

\begin{thm}
\label{thm-bl-b}
Let $(X,\gamma_X)$ and $f:X'\rightarrow X$
be as in Theorem \ref{thm-bl}. Then
\begin{equation}
\taub\big(X',f^*\gamma_X\big) = \taub\big(X,\gamma_X\big) \;.
\end{equation}
\end{thm}
\begin{proof}
This is a direct consequence of Lemma \ref{lem-bl-b} and Theorem \ref{thm-tau-b}.
\end{proof}

\section{Extension to the singular cases}
\label{ch-singular}
We extend the theory of BCOV invariants to Calabi--Yau varieties with mild singularities.
In this section, $X$ is a normal projective complex variety of dimension $n$.

\subsection{Definitions and basic properties}

Recall that a variety $X$ is called $\Q$-\emph{Gorenstein} if the canonical divisor $K_{X}$ is $\Q$-Cartier,
i.e., there exists a positive integer $d$ such that $dK_{X}$ is a Cartier divisor.
The minimal value of such $d$ is called the \emph{index} of $X$.

\begin{defn}[Canonical and KLT singularities {\cite[Definition 2.34]{KollarMori}}]
\label{def-KLT}
Let $X$ be a $\Q$-Gorenstein normal variety
and let $f\colon X'\to X$ be a log-resolution,
i.e., $f$ is a resolution of singularities and the support of the exceptional divisor $E = \bigcup_{j=1}^l E_j$ is simple normal crossing.
Write the equality of $\Q$-divisors
\begin{equation}
K_{X'/X} = \sum_{j=1}^l a_jE_j\;,
\end{equation}
where $K_{X'/X}$ is the relative canonical divisor $K_{X'}-\frac{1}{d}f^*{(dK_X)}$
for any positive integer $d$ divisible by the index of $X$,
where $K_{X'}$ is such that $f_*K_{X'}=K_X$.
We say that $X$ has \emph{canonical} singularities if $a_j\geqslant 0$ for all $j$. Similarly, $X$ is said to have \emph{Kawamata log terminal} (KLT) singularities if $a_j>-1$ for all $j$.
The coefficients $a_j\in\Q$ are called \emph{discrepancy numbers}.
The definitions are independent of the log-resolution $f$.
\end{defn}

\begin{defn}
\label{def-KLTCY}
A \emph{canonical (resp.~KLT) Calabi--Yau} variety is a $\Q$-Gorenstein normal projective complex variety $X$ with canonical (resp.~KLT) singularities such that $K_{X}\sim_{\Q}0$, where $\sim_{\Q}$ is the linear equivalence relation for $\Q$-divisors.
\end{defn}

Let us record the following basic result.
\begin{prop}[Integrability of volume form {\cite[Thm.~2.1]{Sakai77}, \cite[Prop.~1.17]{MR1983959}}]
\label{prop-Integrability}
Let $X$ be an $n$-dimensional variety with KLT singularities.
Let $d$ be a positive integer divisible by the index of $X$.
Then for any $\gamma\in H^0(X, \mathcal{O}_{X}(dK_{X}))$,
the integral
\begin{equation}
\int_{X_{\operatorname{reg}}} \big|\gamma\overline{\gamma}\big|^{1/d}
\end{equation}
converges,
where $X_{\operatorname{reg}}$ is the regular part of $X$
and $\big|\gamma\overline{\gamma}\big|^{1/d}$ is the unique positive volume form on $X$
whose $d$-th power equals to $i^{n^2}\gamma\wedge\overline{\gamma}$.
\end{prop}

We extend the birational BCOV invariant studied in \textsection \ref{ch-birat} to varieties with KLT singularities,
equipped with a pluricanonical form or a pluricanonical \emph{effective} divisor.
\begin{defn}
\label{def-taub-KLT}
Let $X$ be an $n$-dimensional variety with KLT singularities.
Let $d\in \N_{>0}$ divisible by the index of $X$.
Let  $D\in |dK_{X}|$.
Let $f\colon X'\to X$ be a resolution of singularities such that
$\widetilde{D}\cup E$ is of simple normal crossing support,
where $\widetilde{D}$ is the strict transform of $D$ and $E$ is the exceptional divisor.
For any $\gamma\in H^0(X,\mathcal{O}_X(dK_X))$ such that $D=\operatorname{div}(\gamma)$,
we define
\begin{equation}
\label{eq1-def-taub-KLT}
\taub\big(X,\gamma\big) := \taub\big(X',f^*\gamma\big) \;.
\end{equation}
Here the right hand side is defined in Definition \ref{def-tau-b}.
Note that $X$ having KLT singularities implies that $(X', f^*\gamma)$ is indeed a $d$-Calabi--Yau pair (i.e. Condition ($\star_{d}$) is verified).
We also define
\begin{equation}
\label{eq2-def-taub-KLT}
\taub\big(X,D\big) := \taub\big(X,\gamma\big)
+ \frac{\chi_d(X,D)}{12} \log \left(\big(2\pi\big)^{-n}\int_{X_{\operatorname{reg}}}\big|\gamma\overline{\gamma}\big|^{1/d}\right) \;.
\end{equation}
Note that the integral on the right hand side converges by Proposition \ref{prop-Integrability}.
\end{defn}

In the next two lemmas,
we show that $\taub\big(X,\gamma\big)$ and $\taub\big(X,D\big)$ are well-defined.

\begin{lemme} The quantity
$\taub\big(X, \gamma\big)$ is independent of the resolution $f$.
\end{lemme}
\begin{proof}
As any two resolutions are dominated by a third one,
it suffices to show that
for a further blow-up $g\colon X''\to X'$ satisfying the same properties,
we have
\begin{equation}
\taub\big(X', f^*\gamma\big) = \taub\big(X'', g^*f^*(\gamma)\big) \;.
\end{equation}
But this follows from Theorem \ref{thm-bl-b}.
\end{proof}

\begin{lemme}
For any $z\in \C^*$,
the following identity holds,
\begin{equation}
\taub\big(X,z\gamma\big) = \taub\big(X,\gamma\big) - \frac{\chi_d(X,D)}{12}\log|z|^{2/d} \;.
\end{equation}
Consequently, $\taub\big(X,D\big)$ is independent of $\gamma$.
\end{lemme}
\begin{proof}
This follows directly from \cite[Proposition 3.4]{z2}.
\end{proof}

In the sequel, we use the same notation as in \textsection \ref{subsect:MotInt}.

\begin{defn}
\label{def-GorensteinVol}
Let $X$ be a variety with KLT singularities.
In the situation of Definition \ref{def-KLT},
the \emph{Gorenstein volume} of $X$ is by definition
\begin{equation}
\mu^{\operatorname{Gor}}(X) :=
\sum_{J\subseteq\{1,\dots,l\}} \mathbb{L}^{|J|-n} \left(\prod_{j\in J} \frac{1-\mathbb{L}^{a_j}}{\mathbb{L}^{a_j+1}-1}\right)[E_{J}]
\in \widehat{\mathcal{M}} \;.
\end{equation}
In other words,
for any $d\in \N_{>0}$ divisible by the index of $X$,
$\mu^{\operatorname{Gor}}(X)$ is equal to $F_{d}(X', dK_{X'/X})$ defined in \eqref{eq-MotIntDef}.
By Theorem \ref{thm-MotInt},
the definition is independent of the choice of $d$ and the log-resolution $X'$.
Note that when $X$ is smooth,
$\mu^{\operatorname{Gor}}(X) = \mathbb{L}^{-n}[X]$.
\end{defn}

Let $P_t: K_{0}(\Var_{\C}) \rightarrow \Z[t]$ be the ring homomorphism sending a smooth projective variety to its Poincar{\'e} polynomial,
which extends to a ring homomorphism $P_t: \widehat{\mathcal{M}}\to \Z[[t, t^{-1}]]$ (cf. \cite[\S 3.4.7]{MotIntBook}).

\begin{defn}
\label{def-stringy-inv}
Let $X$ be an $n$-dimensional variety with KLT singularities.
The \emph{stringy Poincar\'e polynomial} of $X$ is defined as
\begin{equation}
\label{eq-def-stringy}
P_{t}(X) := P_{t}(\mathbb{L}^{n}\mu^{\operatorname{Gor}}(X)) \;.
\end{equation}
Following Batyrev \cite{Batyrev98}, we define the \emph{stringy Betti numbers} of $X$ as the coefficients of $P_{t}(X)$.
The quantities $\chi(X)$, $\chi'(X)$ and $\chi''(X)$ are defined by the same formulas \eqref{eq-Def-Chi}--\eqref{eq-Def-Chi''}
with Betti numbers replaced by stringy Betti numbers.
If $X$ admits a crepant resolution $Y$,
the stringy invariants of $X$ equal to the corresponding invariants of $Y$.
\end{defn}

\begin{defn}
\label{def-BCOV-KLT}
Let $X$ be an $n$-dimensional KLT Calabi--Yau variety (Definition \ref{def-KLTCY}).
We define the (stringy) BCOV invariant of $X$ as
\begin{equation}
\label{eq-def-stringy-BCOV}
\tau(X) := \taub\big(X,\emptyset\big) + \frac{1}{2} \tau(\CP^1) \chi'(X)
+ \Big( \frac{1}{2}\tau(\CP^2) - \frac{3}{4}\tau(\CP^1) \Big) \chi''(X) \;,
\end{equation}
where $d\in \N_{>0}$ is divisible by the index of $X$ and such that $|dK_{X}|\neq \emptyset$,
$\taub\big(X,\emptyset\big)$ is defined in  \eqref{eq2-def-taub-KLT} (it is independent of $d$ by \cite[Proposition 3.3]{z2}),
$\chi'(X)$ and $\chi''(X)$ are the stringy invariants introduced in Definition \ref{def-stringy-inv}.
By \eqref{eq4-proof-thm-birat},
we recover the BCOV invariant when $X$ is smooth.
\end{defn}

\begin{rem}\label{rmk-OrbifoldBCOV}
Quotient singularities form one of the most important instances of KLT singularities,
see \cite[Proposition 5.20]{KollarMori}.
In particular,
(complex effective) orbifolds,
or equivalently, V-manifolds in the sense of Satake \cite{Satake56, Satake57},
are KLT.
A compact K\"ahler orbifold $X$ is called Calabi--Yau if $c_1(X)=0 \in H^2(X, \R)$.
On one hand, thanks to the orbifold version of the Beauville--Bogomolov decomposition theorem,
due to Campana \cite{Campana04} and Fujiki \cite{Fujiki83}, Calabi--Yau
orbifolds have torsion canonical divisor.
Therefore Calabi--Yau orbifolds are special cases of KLT Calabi--Yau varieties in the sense of Definition \ref{def-KLTCY},
hence their BCOV invariants can be defined as in Definition \ref{def-BCOV-KLT}.
On the other hand,
the Quillen metric can be extended to orbifolds (see Ma \cite[\textsection 2]{Ma-orbifold})
and enjoys similar properties as in the smooth case (see Ma \cite{Ma-orbifold, Ma-orbifold2}).
Hence the definition of BCOV invariant (see Definition \ref{def-BCOV-Pair}) can be directly extended to Calabi--Yau orbifolds. We plan to
compare the two definitions in the future.
\end{rem}

\subsection{Curvature formula}
We extend the curvature formula \cite[Theorem 0.4]{z2} to
locally trivial families (in the sense of Flenner--Kosarew \cite[Page 627]{FlennerKosarew87}) of KLT Calabi--Yau varieties.

\begin{defn}
\label{def-LocTrivial}
Let $S$ be a complex manifold and $\mathcal{X}$ a complex space.
Let $\pi\colon\mathcal{X}\to S$ be a flat and proper morphism,
viewed as a family of complex spaces $\big(X_t := \pi^{-1}(t)\big)_{t\in S}$.
The family $\pi$ is called \emph{locally trivial}
if for any $t\in S$ and any $x\in X_t$,
there are analytic open neighborhoods $\Delta\subset S$ of $t$ and $\mathcal{U}\subset\pi^{-1}(\Delta)\subset \mathcal{X}$ of $x$
such that we have a $\Delta$-isomorphism $(\mathcal{U}\cap X_t) \times \Delta \simeq \mathcal{U}$.
\end{defn}

Locally trivial families admit (strong) simultaneous resolution.

\begin{lemme}[Simultaneous resolution {\cite[Lemma 4.9 and Remark 4.10]{BakkerLehn}}]
\label{lem-SimultRes}
Let $\pi: \mathcal{X}\to S$ be a locally trivial family.
Then there is a proper bimeromorphic $S$-morphism $\mathcal{Y}\to\mathcal{X}$,
which is a composition of blow-ups along locally trivial centers
which are smooth over $S$ and disjoint from the smooth locus of $\pi$,
such that the composed map $\pi': \mathcal{Y}\to S$ is a submersion and for any $t\in S$,
the map $Y_t := (\pi')^{-1}(t) \to X_t$ is a log resolution.
\end{lemme}

In the sequel,
a Hodge structure (resp. variation of Hodge structures)
is a finite direct sum of pure polarizable $\Q$-Hodge structures (resp. variations of pure polarizable $\Q$-Hodge structures),
possibly of different weights.
Following \cite[Proposition 2.8]{fl}, \cite[(5.6)]{efm} and \cite[(0.6)]{z},
we make the following definition:

\begin{defn}
Let $S$ be a complex manifold.
Let $(\mathbb{H}, \mathcal{H}, F^{\bullet})$ be a variation of Hodge structures over $S$,
where $\mathbb{H}$ is a local system,
$\mathcal{H}=\mathbb{H}\otimes \mathcal{O}_{S}$ and $F^{\bullet}$ is a Hodge filtration on $\mathcal{H}$.
For any $k\in \Z$,
denote by $\mathbb{H}^k$ the weight-$k$ part of the variation.
For any $p, q\in \Z$,
denote $\mathcal{H}^{p,q}:=\Gr_{F}^{p}\mathcal{H}^{p+q}$,
which we view as a holomorphic vector bundle over $S$.
The \emph{Hodge form} of the variation is the following  $(1, 1)$-form on $S$:
\begin{equation}
\omega_{\mathbb{H}} := \frac{1}{2} \sum_{p, q} (-1)^{p+q}(p-q)c_{1}(\mathcal{H}^{p,q}, g^{\mathcal{H}^{p,q}})
\in A^{1,1}(S) \;.
\end{equation}
Here $g^{\mathcal{H}^{p,q}}$ is a Hermitian metric on $ \mathcal{H}^{p,q}$ such that $g^{\mathcal{H}^{p,q}}(u,u)=g^{\mathcal{H}^{q,p}}(\overline{u},\overline{u})$.
One can show (cf. \cite[Proposition 1.1]{z}) that
$\omega_{\mathbb{H}}$ is independent of the Hermitian metrics $g^{\mathcal{H}^{p,q}}$.
Clearly,
$\omega$ is additive with respect to short exact sequences of variations of Hodge structures,
hence it gives rise to a group homomorphism:
\begin{equation}
\label{eq-Hodge-form}
\omega: K_{0}(\VHS_S)\to A^{1,1}(S) \;,
\end{equation}
where $\VHS_S$ is the category of variations of Hodge structures over $S$.
\end{defn}

On the other hand,
the Hodge realization can be performed in the relative setting:
given a complex variety $S$, there is a group homomorphism
\begin{align}
\begin{split}
\chi_{\Hdg, S}\colon  K_{0}(\Var_{S})&\to K_{0}(\MHM_{S})\\
(\pi\colon \mathcal{X}\to S) & \mapsto \operatorname{R\pi_{!}}\Q_{\mathcal{X}} \;,
\end{split}
\end{align}
where $\MHM_{S}$ is the category of mixed Hodge modules over $S$.
However,
note that if we start with a smooth proper morphism  $\pi: \mathcal{X}\to S$,
then the image $\operatorname{R\pi}_{!}\Q_{\mathcal{X}}$ lies in $K_{0}(\VHS_{S})$,
the subgroup generated by variations of Hodge structures over $S$.

Now let $\pi: \mathcal{X}\to S$ be a locally trivial family of KLT Calabi--Yau varieties.
The \emph{Gorenstein volume} in Definition \ref{def-GorensteinVol} can be extended as follows.
Take a simultaneous resolution (see Lemma \ref{lem-SimultRes}) $f\colon \mathcal{X'} \to \mathcal{X}$
with simple normal crossing exceptional divisor $\mathcal{E}=\bigcup_{j=1}^{l} \mathcal{E}_j$,
which is locally trivial over $S$.
For any $1\leqslant j\leqslant l$,
let $a_j\in \Q_{>-1}$ be the discrepancy number of the resolution $f$ along $\mathcal{E}_j$.
Then define the \emph{relative Gorenstein volume}
\begin{equation}
\mu^{\Gor}(\mathcal{X}/S) := \sum_{J\subseteq \{1,\dots,l\}} \mathbb{L}^{|J|-n}
\left(\prod_{j\in J} \frac{1-\mathbb{L}^{a_j}}{\mathbb{L}^{a_j+1}-1}\right)[\mathcal{E}_{J}/S]
\in \widehat{\mathcal{M}_{S}}[\mathbb{L}^{1/d}] \;,
\end{equation}
where $\widehat{\mathcal{M}_{S}}$ is the completion of $K_{0}(\Var_{S})[\mathbb{L}^{-1}]$ with respect to the dimension filtration.
Similarly to Theorem \ref{thm-MotInt},
$\mu^{\Gor}(\mathcal{X}/S)$ is independent of the resolution $f$.

Taking the Hodge realization
and denoting $\mathbf{L}:=\Q_{S}(-1)$ the Lefschetz variation of Hodge structure over $S$,
we get
\begin{equation}
\label{eq-Stringy-HodgeStructure}
\chi_{\Hdg,S}\left(\mu^{\Gor}(\mathcal{X}/S)\right) = \sum_{J\subseteq\{1,\dots,l\}} \mathbf{L}^{|J|-n}
\left(\prod_{j\in J} \frac{1-\mathbf{L}^{a_j}}{\mathbf{L}^{a_j+1}-1}\right)H^\bullet(\mathcal{E}_J/S) \;,
\end{equation}
where $H^\bullet(\mathcal{E}_J/S)$ denotes the variation of Hodge structures
$\operatorname{R\pi}_{J,*}\Q_{\mathcal{E}_{J}} := \bigoplus_k \operatorname{R}^k \pi_{J, *}\Q_{\mathcal{E}_J}$,
where $\pi_J: \mathcal{E}_J\to S$ is the natural projection.

\begin{defn}
\label{def-stringy-Hodge-form}
Let $\pi: \mathcal{X}\to S$ be as above.
Its (stringy) \emph{Hodge form},
denoted by $\omega_{\Hdg, \mathcal{X}/S}$,
is the image of $\chi_{\Hdg, S}\left(\mu^{\Gor}(\mathcal{X}/S)\right)$ via \eqref{eq-Hodge-form}.
\end{defn}

\begin{lemme}
\label{lem-stringy-Hodge-form}
Notation is as before.
Taking a simultaneous resolution $f\colon \mathcal{X}'\to \mathcal{X}$ as above,
the Hodge form can be computed as
\begin{equation}
\omega_{\Hdg,\mathcal{X}/S} =
\sum_{J\subseteq\{1,\dots,l\}} \left(\prod_{j\in J} \frac{-a_j}{a_j+1}\right)\omega_{H^\bullet(\mathcal{E}_J/S)} \;.
\end{equation}
\end{lemme}
\begin{proof}
It suffices to apply the homomorphism \eqref{eq-Hodge-form} to the right hand side of \eqref{eq-Stringy-HodgeStructure},
and use the fact that $\omega_{H^\bullet} = \omega_{\mathbf{L}H^\bullet}$.
\end{proof}

\begin{defn}
\label{def-WP-form}
Let $\pi: \mathcal{X}\to S$ be as above and $d$ an integer divisible by the index of fibers.
The \emph{Weil--Petersson form} of $\pi: \mathcal{X}\to S$,
denoted by $\omega_{\mathrm{WP}, \mathcal{X}/S}$,
is defined as follows:
for any open subset $U\subseteq S$
and any nowhere vanishing section
$\gamma\in H^0(U, \pi_*\mathcal{O}(dK_{\mathcal{X}/S}))$ viewed as
holomorphic family $\big(\gamma_s\big)_{s\in U}$,
we define
\begin{equation}
\label{eqn:WPdefinition}
\omega_{\mathrm{WP},\mathcal{X}/S}\big|_U =
- \frac{\overline{\partial}\partial}{2\pi i} \log\int_{X_{\operatorname{reg}}} \big|\gamma\overline{\gamma}\big|^{1/d} \;,
\end{equation}
where $\int_{X_{\operatorname{reg}}} \big|\gamma\overline{\gamma}\big|^{1/d}$ is the function
$s \mapsto \int_{X_{s,\operatorname{reg}}} \big|\gamma_s\overline{\gamma_s}\big|^{1/d}$
whose convergence is guaranteed by Proposition \ref{prop-Integrability}. Thanks to the next lemma, the  partial derivatives make sense in \eqref{eqn:WPdefinition}.
\end{defn}

\begin{lemme}
\label{lem-int-smooth}
The function $S \ni s \mapsto \int_{X_{s,\operatorname{reg}}} \big|\gamma_s\overline{\gamma_s}\big|^{1/d}$ is smooth.
\end{lemme}
\begin{proof}
The question is local on $S$.
We may assume that $\gamma\in H^0(S, \pi_*\mathcal{O}(dK_{\mathcal{X}/S}))$, i.e. a holomorphic family of $d$-canonical forms on the fibers.
By Lemma \ref{lem-SimultRes},
there is a simultaneous log-resolution $f\colon\mathcal{X}'\to \mathcal{X}$ with simple normal crossing exceptional divisor $\mathcal{E}=\bigcup_{j=1}^{l} \mathcal{E}_j$,
which is also locally trivial over $S$.
Set $\gamma' = f^*\gamma$,
which is a family of meromorphic $d$-canonical forms on the fibers of $\mathcal{X}'/S$.
We obviously have
\begin{equation}
\int_{X_{s,\operatorname{reg}}} \big|\gamma_s\overline{\gamma_s}\big|^{1/d} = \int_{X'_{s}} \big|\gamma'_s\overline{\gamma'_s}\big|^{1/d} \;.
\end{equation}
Let
\begin{equation}
\Big(\alpha_k: \C^m \supseteq U_k \hookrightarrow  S \Big)_{k=1,\cdots,p} \;,\hspace{5mm}
\Big(\beta_k: \C^m \times \C^n \supseteq V_k \hookrightarrow  \mathcal{X}'\Big)_{k=1,\cdots,p}
\end{equation}
be holomorphic local charts such that
\begin{itemize}
\item[-] for each $k$,
we have $(\pi \circ f)(\beta_k(V_k)) = \alpha_k(U_k)$;
\item[-] for each $k$,
the map $\alpha_k^{-1} \circ \pi \circ f \circ \beta_k$ is given by the canonical projection $\C^m \times \C^n \rightarrow \C^m$;
\item[-] for each $k$ and each $\mathcal{E}_j$,
either $\beta_k^{-1}(\mathcal{E}_j) = \emptyset$ or $\beta_k^{-1}(\mathcal{E}_j) = \big\{ (z_1,\cdots,z_m,w_1,\cdots,w_n) \in V_k \subseteq \C^m \times \C^n \;:\; w_q = 0 \big\}$,
where $q\in\{1,\cdots,n\}$ depends on $j,k$.
\end{itemize}
Let
\begin{equation}
\Big(\eta_k: \beta_k(V_k) \rightarrow [0,1] \Big)_{k=1,\cdots,p}
\end{equation}
be a partition of unity, i.e., each $\eta_k$ is of compact support and $\eta_1+\cdots+\eta_p = 1$.
We denote $\eta_k\big|_{X_s'} = \eta_{k,s}$.
It is sufficient to show that for each $k$, the function
\begin{equation}
\label{eq1-pf-lem-int-smooth}
s \mapsto \int_{X_s'} \eta_{k,s} \big|\gamma'_s\overline{\gamma'_s}\big|^{1/d}
\end{equation}
is smooth.
On the other hand,
by the KLT condition,
the order of poles of $\gamma'$ along each $\mathcal{E}_j$ is at most $d-1$.
Hence we have
\begin{equation}
\beta_k^*\gamma' = \frac{g}{w_1^{d-1} \cdots w_n^{d-1}} \big(\d w_1\wedge\cdots \wedge \d w_n\big)^d \;,
\end{equation}
where $g$ is a holomorphic function of $z_1,\cdots,z_m,w_1,\cdots,w_n$.
Then a straightforward computation shows that the integration \eqref{eq1-pf-lem-int-smooth} results in a smooth function in $z_1, \cdots, z_m$.
\end{proof}

We are now ready to prove Theorem \ref{thm-main-curvature},
which we state again for convenience.

\begin{thm}
\label{thm-curvature-KLT}
Let $\pi\colon \mathcal{X}\to S$ be a flat family of KLT Calabi--Yau varieties.
Assume that $\pi$ is locally trivial.
Then the function  $\tau(\mathcal{X}/S): S \ni t  \mapsto \tau(X_t)$ is smooth,
where $\tau$ is the BCOV invariant in Definition \ref{def-BCOV-KLT}.
Moreover,
we have
\begin{equation}
\label{eq-thm-curvature-KLT}
\frac{\overline{\partial}\partial}{2\pi i}\tau(\mathcal{X}/S) =
\omega_{\mathrm{Hdg},\mathcal{X}/S} - \frac{\chi(X)}{12} \omega_{\mathrm{WP},\mathcal{X}/S} \;,
\end{equation}
where $\chi(X)$  is the stringy Euler characteristic of a fiber of $\pi$,
$\omega_{\mathcal{X}/S}$ is the Hodge form in Definition \ref{def-stringy-Hodge-form},
and $\omega_{\mathrm{WP},\mathcal{X}/S}$ is the Weil--Petersson form in Definition \ref{def-WP-form}.
\end{thm}
\begin{proof}
The smoothness of the function $\tau_{\mathcal{X}/S}$ comes from the existence of simultaneous resolutions.

Let $d\in \N_{>0}$ be such that $|dK_{X_{t}}|\neq \emptyset$.
Notation being as before,
we have
\begin{align}
\label{eq1-pf-thm-curvature-KLT}
\begin{split}
\frac{\overline{\partial}\partial}{2\pi i} \tau(\mathcal{X}/S)
& = \frac{\overline{\partial}\partial}{2\pi i} \taub\big(\mathcal{X}/S, \emptyset\big) \\
& = \frac{\overline{\partial}\partial}{2\pi i} \taub\big(\mathcal{X}'/S, \sum_j da_j\mathcal{E}_j/S\big)
= \frac{\overline{\partial}\partial}{2\pi i} \tau_d\big(\mathcal{X}'/S, \sum_jda_j\mathcal{E}_j/S\big) \;,
\end{split}
\end{align}
Here the first and the third equalities come from the fact that
$\taub-\tau_d$ consists of topological invariants,
which are constant in locally trivial families.
The second equality comes from \eqref{eq1-def-taub-KLT}.
By Definition \ref{def-taub-KLT}, Definition \ref{def-BCOV-KLT} and \cite[Theorem 0.4]{z2},
we have
\begin{align}
\label{eq2-pf-thm-curvature-KLT}
\begin{split}
& \frac{\overline{\partial}\partial}{2\pi i} \tau_d\big(\mathcal{X}'/S, \sum_jda_j\mathcal{E}_j/S\big) \\
& = \sum_{J\subseteq\{1,\dots,l\}} \left(\prod_{j\in J} \frac{-a_j}{a_j+1}\right) \omega_{H^\bullet(\mathcal{E}_J/S)}
+ \frac{1}{12} \chi_d\big(X',\sum_jda_jE_j\big) \frac{\overline{\partial}\partial}{2\pi i} \int_{X'}\big|\gamma'\overline{\gamma'}\big|^{1/d} \;.
\end{split}
\end{align}
By Lemma \ref{lem-stringy-Hodge-form},
we have
\begin{equation}
\label{eq3-pf-thm-curvature-KLT}
\sum_{J\subseteq\{1,\dots,l\}} \left(\prod_{j\in J} \frac{-a_j}{a_j+1}\right) \omega_{H^\bullet(\mathcal{E}_J/S)}
= \omega_{\Hdg,\mathcal{X}/S} \;.
\end{equation}
Using Definition \ref{def-WP-form},
we can show that
\begin{equation}
\label{eq4-pf-thm-curvature-KLT}
\frac{\overline{\partial}\partial}{2\pi i} \int_{X'}\big|\gamma'\overline{\gamma'}\big|^{1/d} =
\frac{\overline{\partial}\partial}{2\pi i} \int_{X_\mathrm{reg}}\big|\gamma\overline{\gamma}\big|^{1/d} =
- \omega_{\mathrm{WP},\mathcal{X}/S} \;.
\end{equation}
By the definition of the stringy Euler characteristic,
we have
\begin{equation}
\label{eq5-pf-thm-curvature-KLT}
\chi_d\big(X',\sum_jda_jE_j\big) = \chi(X) \;.
\end{equation}
Combining \eqref{eq1-pf-thm-curvature-KLT}--\eqref{eq5-pf-thm-curvature-KLT},
we obtain \eqref{eq-thm-curvature-KLT}.
\end{proof}

\section{Birational invariance}
\label{ch-proof}

In this section, we prove our main result Theorem \ref{thm-main-singular}.
Although it clearly contains Theorem \ref{thm-main-BirInv} as the smooth case, we nevertheless choose to give first the proof in this special case to highlight the main idea:
\begin{proof}[Proof of Theorem \ref{thm-main-BirInv}]
	Let $X$ and $X'$ be $n$-dimensional birationally isomorphic Calabi--Yau manifolds.
	By the weak factorization theorem of Abramovich, Karu, Matsuki, and W{\l}odarczyk \cite[Theorem 0.3.1]{akmw} (see also \cite{MR2013783}),
	there is a sequence of blow-ups and blow-downs along smooth centers:
	\begin{equation}
	\label{eq-Seqence-proof-thm-birat}
	X=X_{0}\dasharrow X_{1}\dasharrow \dots \dasharrow X_{r-1} \dasharrow X_{r}=X' \;,
	\end{equation}
	such that for each $0\leqslant i\leqslant r$,
	the unique canonical divisor $D_i \in \big|K_{X_i}\big|$ is of simple normal crossing support.
	For each $i$,
	let $\gamma_i\in H^0\big(X_i, \mathcal{O}_{X_i}(D_i)\big) = H^0\big(X_i, \mathcal{O}_{X_{i}}(K_{X_i})\big)$
	be such that $\int_{X_i} \big|\gamma_i\overline{\gamma_i}\big| = (2\pi)^n$.
	By Theorem \ref{thm-bl-b} and \cite[Proposition 3.4]{z2},
	we have
	\begin{equation}
	\label{eq2-proof-thm-birat}
	\tau_1^{\mathrm{bir}}\big(X_i,\gamma_i\big) = \tau_1^{\mathrm{bir}}\big(X_{i+1},\gamma_{i+1}\big)
	\end{equation}
	for all $i$. Hence
	\[\tau_1^{\mathrm{bir}}\big(X , \gamma_0\big) = \tau_1^{\mathrm{bir}}\big(X' ,\gamma_r\big)\]
 	Combining this with \eqref{eq4-proof-thm-birat}, we see that in order to prove $\tau(X)=\tau(X')$, it is enough to show that $\chi'(X)=\chi'(X')$ and $\chi''(X)=\chi''(X')$. However, $\chi'(\cdot)$ and $\chi''(\cdot)$, defined in \eqref{eq-Def-Chi'} and \eqref{eq-Def-Chi''} respectively, are certain linear combinations of Betti numbers, hence are birational invariant for Calabi--Yau manifolds by Batyrev \cite{Batyrev99}.
\end{proof}

Now let us proceed to the proof in the general case.

\begin{proof}[Proof of Theorem \ref{thm-main-singular}]
Recall that $\mu^{\operatorname{Gor}}(\cdot)$ was defined in Definition \ref{def-GorensteinVol}.
By a result of Kontsevich \cite{Kontsevich} and Yasuda \cite[Proposition 1.2]{Yasuda},
which extends a result of Batyrev \cite{Batyrev99},
we have $\mu^{\operatorname{Gor}}(X)=\mu^{\operatorname{Gor}}(X')$.
Then,
by Definition \ref{def-stringy-inv},
we have
\begin{equation}
\label{eq1-proof-thm-birat-KLT}
\chi'(X)=\chi'(X') \;,\hspace{5mm}
\chi''(X)=\chi''(X') \;.
\end{equation}

Let $f\colon \widetilde{X}\to X$ and $f'\colon \widetilde{X'}\to X'$ be log-resolutions. Let $d\in \N_{>0}$ be such that $dK_{X}\sim 0$ and $dK_{X'}\sim 0$ as Cartier divisors.
Note that the hypothesis that $X$ and $X'$ have canonical singularities implies that any smooth birational model of $X$ and $X'$ admits a unique $d$-canonical divisor.

By Abramovich--Temkin \cite[Theorem 1.2.1 and \S 1.6]{AbramovichTemkin19},
there is a sequence of blow-ups and blow-downs along smooth centers:
\begin{equation}
\label{eq-Seqence-proof-thm-birat-KLT}
\widetilde{X}=X_{0}\dasharrow X_{1}\dasharrow \dots \dasharrow X_{r-1} \dasharrow X_{r}=\widetilde{X'} \;,
\end{equation}
such that for each $0\leqslant i\leqslant r$,
the (unique) $d$-canonical divisor $D_i \in \big|dK_{X_i}\big|$ is of simple normal crossing support.
For each $i$,
let $\gamma_i\in H^0\big(X_i, \mathcal{O}_{X_i}(D_i)\big) = H^0\big(X_i, \mathcal{O}_{X_{i}}(dK_{X_i})\big)$
be such that $\int_{X_i} \big|\gamma_i\overline{\gamma_i}\big|^{1/d} = (2\pi)^n$.
By Theorem \ref{thm-bl-b},
we have
\begin{equation}
\label{eq2-proof-thm-birat-KLT}
\taub\big(X_i,\gamma_i\big) = \taub\big(X_{i+1},\gamma_{i+1}\big)
\end{equation}
for all $i$.
Let $\gamma \in H^0\big(X, \mathcal{O}_{X}(dK_X)\big)$ and $\gamma' \in H^0\big(X', \mathcal{O}_{X'}(dK_{X'})\big)$
be such that $\int_X \big|\gamma\overline{\gamma}\big|^{1/d} = \int_{X'} \big|\gamma'\overline{\gamma'}\big|^{1/d} = (2\pi)^n$.
By Definition \ref{def-taub-KLT} and \cite[Proposition 3.4]{z2},
we have
\begin{align}
\label{eq3-proof-thm-birat-KLT}
\begin{split}
\taub\big(X,\emptyset\big) & = \taub\big(X,\gamma\big) = \taub\big(X_0,\gamma_0\big) \;,\\
\taub\big(X',\emptyset\big) & = \taub\big(X',\gamma'\big) = \taub\big(X_r,\gamma_r\big) \;.
\end{split}
\end{align}
From Definition \ref{def-BCOV-KLT},
\eqref{eq1-proof-thm-birat-KLT}, \eqref{eq2-proof-thm-birat-KLT} and \eqref{eq3-proof-thm-birat-KLT},
we obtain $\tau(X)=\tau(X')$.
\end{proof}

\bigskip

\bibliographystyle{abbrv}
\bibliography{bcovfuzhang}

\emph{Lie Fu}, \texttt{lie.fu@math.unistra.fr} \\
\textsc{Institut de recherche math\'ematique avanc\'ee (IRMA), Universit\'e de Strasbourg, 7 rue Ren\'e--Descartes, 67084 Strasbourg Cedex, France}\\

\emph{Yeping Zhang}, \texttt{yepingzhang1987@hotmail.com} \\
\textsc{Shanghai, China}

\end{document}